\documentclass[a4paper,12pt]{amsart}
\usepackage{amsmath,amsthm,amsfonts,amssymb,amscd}
\usepackage[utf8x]{inputenc}
\usepackage[T1]{fontenc}
\usepackage{csquotes}
\usepackage[all]{xy}
\usepackage{float}
\usepackage{listings}
\usepackage{xcolor}
\definecolor{codegreen}{rgb}{0,0.6,0}
\definecolor{codegray}{rgb}{0.5,0.5,0.5}
\definecolor{codepurple}{rgb}{0.58,0,0.82}
\definecolor{backcolour}{rgb}{0.95,0.95,0.92}

\lstdefinestyle{mystyle}{
    backgroundcolor=\color{backcolour},
    commentstyle=\color{codegreen},
    keywordstyle=\color{magenta},
    numberstyle=\tiny\color{codegray},
    stringstyle=\color{codepurple},
    basicstyle=\ttfamily\footnotesize,
    breakatwhitespace=false,
    breaklines=true,
    captionpos=b,
    keepspaces=true,
    numbers=left,
    numbersep=5pt,
    showspaces=false,
    showstringspaces=false,
    showtabs=false,
    tabsize=2
}

\usepackage{xcolor}
\usepackage{comment}

\newtheorem{theorem}{Theorem}

\newtheorem{Example}{Example}[section]

\newtheorem{definition}[theorem]{Definition}
\newtheorem{example}{Exemplo}[section]
\newtheorem{lemma}[theorem]{Lema}
\makeatletter
\def\author@andify{%
  \nxandlist {\unskip ,\penalty-1 \space\ignorespaces}%
    {\unskip {} \@@and~}%
    {\unskip \penalty-2 \space \@@and~}%
}
\makeatother

\title[INTERNAL EXACT CONTROLLABILITY FOR NAGDHI SHELL]{INTERNAL EXACT CONTROLLABILITY FOR NAGDHI SHELL}

\author{Alexis Rodriguez Carranza}
\address{A. Rodriguez. National University of Trujillo}
\email{arodriguezca@unitru.edu.pe}

\author{Jose Luis Ponte Bejarano}
\address{J.Ponte. National University of Trujillo}
\email{jponteb@unitru.edu.pe}

\author{Juan Carlos Ponte Bejarano}
\address{C. Ponte. National University of Trujillo}
\email{jponte@unitru.edu.pe}


\begin{document}

\begin{abstract}
In this work we study the exponential stability of the energy associated with a Naghdi shell model with localized internal dissipation. Using several tools from Riemannian Geometry we show the well possedness of the model, via semigroup theory, and obtain observability inequalities which allows to prove the exponential decay of the total energy. As a consequence then we use Russell Principle for obtain exact controllability.
\end{abstract}

\maketitle

\tableofcontents

\section{Introduction and summary}

Let us denote by $S$ a two dimensional smoooth Riemannian Manifold with the induced metric of $I\!\!R^3$ and inner product denoted by $g$ or simply $<.,.>$. Recall that this means that for each $p\in S$ we have an inner product $<.,.>$ on the tangent space $T_{p}S$ and this relation is $C^{\infty}$. We will consider $S$ as the middle surface of a thin shell.

Suppose we consider a bounded open region $M$ of $S$ with a smooth boundary $\Gamma=\partial M$. This paper is devoted to consider $M$ as a Naghdi type thin shell. In this dynamical model we study exponential stabilization of the total energy provided we assume an internal localized dissipation acting on $M$.

In the literature, mos of the authors prefer to use the classical geometrical approach while working on properties of the solutions of evolution thin shell. In those situations the middle surface is instead, the image under a smooth map of a two dimensional connected domain de $I\!\!R^2$ and therefore described under just a one coordinate path. Several interesting model for plates or shells with a variable coefficients and other hyperbolic type systems obtained by using traditional geomentry became very difficult to treat mainly due to the explicit presence of Chritoffel symbols $\Gamma^{k}_{ij}(p)$. Interesting aternative was given in the work due to peng fei Yao \cite{Yao} and collaborators about twenty years ago. In \cite{Yao}, used and intrinsic model of the middle surface of a shallow shell as a two dimensional Riemannianan manifold. This approach allows to get satisfactory results using multipliers. The basic idea was iniciated by S. Bochner in \cite{Bochner}. In order to very an identity or a pointwise estimate, it sufficies to do so at each point $p$ relative to a coordinate frame field wich give as more simplifications. The best coordinate system in our case would be the one in wich the symbols $\Gamma^{k}_{ij}(p)$ vanish at the given point $p$. As in \cite{Bochner} the best frame will be given by the so called coordinate system normal at $p$.

This paper is devoted to study exponential stabilization of the total energy associated with a dynamic thin shell equation of Naghdi type in the presence of localized internal dissipation.Using several Yao ideas we show the well posedness of the model, via semigroup theory, and obtain observability inequalities wich allo us to prove the exponential decay of the total energy. As a consequence then we use Russell Principle for obtain exact controllability.

\subsection{Preliminaries and stationary problem}
Let $S$ and $M$ be as in the introduction. In addition we will assume $S$ orientable and a normal field at each $x\in M$ will be denote by $N(x)$. The shell, a body of $I\!\!R^3$ is definied as
\[
S=\{p \in I\!\!R^3,\quad p=x+zN(x),\quad x\in M,\quad |x|<\frac{h}{2}\}
\]
where $h>0$ denotes the, small, thickness off the shell and $z\in I\!\!R$. In Naghdi model, the displacement vector $\xi(p)$ at a point $p\in S$ can be approximated by 
\begin{equation}\label{Eq21}
\xi(p)=\xi_1(p)+z\Psi(x),\quad x\in M,
\end{equation}
where $p=x+zN(x) \xi_1(x)\in I\!\!R^3$ denotes the displacement vector of the middle surface and $\Psi(x)$ captures the rotations of the normal $N(x)$ at each $x\in M$.
Following Naghdi description \cite{Naghdi} we assume that a normal after a deformation may not be again a normal, but, the distance from a point on the normal to the surface remains invariant. As a consequence, the deformations in the direction of the normal can be neglected. This implies in particular that $\xi_1(x)$ and $\Psi(x)$ in (\ref{Eq21}) can be decompose as
\begin{equation}\label{Eq22}
\xi_1(x)=W_1(x)+w_1(x)N(x)
\end{equation}
and 
\begin{equation}\label{Eq23}
\Psi_1(x)=V(x)+w_2(x)N(x)
\end{equation}
where $W_1$ and $V\in \chi(M)$, and $w_1$, $w_2\in C^{\infty}(M)$. Here $\chi(M)$ denotes the set of all vector fieds on $M$. We recall that a vector field is a map wich associates to each point $p\in M$ a vectr $X(p)\in T_p(M)$ wich is the tangential plane of $M$ at $p$.
\\
In order to find a model describing deformations of the middle surface we need to analyse the tensor field (of variation of the metric) using the second and thrid-fundamental forms on the surface $M$. In other words, we consider 
\begin{equation}\label{Eq24}
\Upsilon = \frac{1}{2}(\tilde{g}-g)
\end{equation}

where $g$ and $\tilde{g}$ denote the metric induced on the middle surface before and after the deformation, respectively. $\Upsilon$ is called the strain tensor of the middle surface. Given $x\in M$, we choose a coordinare system normal at $x$, say $\{ E_1(x), E_2(x), E_3(x)=N(x)\}$ wich is a basis of $I\!\!R^3$. Now, we calculate $\tilde{g}(E_i, E_j)-g(E_i, E_j)$ to find, after linearization
\begin{equation}{\label{Eq25}}
\Upsilon(\xi)=\frac{1}{2}(DW_1+D^{*}W_1)+w_1\Pi
\end{equation}
where $\xi=(W_1,V,w_1,w_2)$, $DW_1$ is the covariant differential of $W_1$ of $W_1$ and $D^{*}W_1$ is the transpose of $DW_1$ and $\Pi$ is the second fundamental form of $M$ wich is a $2-$ covariant tensor. $\Upsilon(\xi)$ is called the tensor change of metric(linearized). In a similar way we can deduce the change of curvature tensor(linearized)

\begin{equation}{\label{Eq26}}
\chi_0(\xi)=\frac{1}{2}\{DV+D^{*}V+\Pi(.,DW_1)+\Pi(DW_1,.)\}+w_2\Pi+w_1c
\end{equation}
where $c$ is the third fundamental form on the surface $M$. Also, the tensor wich captures rotations of the normal is given by:

\begin{equation}{\label{Eq27}}
\varphi_0(\xi)=\frac{1}{2}[Dw_1+V-i(W_1)\Pi]
\end{equation}
where $i(W_1)\Pi$ is the interior product of the tensro field $\Pi$ by the vector field $W_1$. It is convenient to write (\ref{Eq25}), (\ref{Eq26}), (\ref{Eq27}) in a more concise way. For example, consider the change of variable,

\begin{equation}{\label{Eq28}}
W_2=V+i(W_1)\Pi
\end{equation}
for $x\in M$. As above we consider the system normal at $x$, $\{E_1(x), E_2(x), E_3(x)=N(x)\}$. Direct calculations give us,

\begin{equation*}
DW_2(E_i,E_j)=DV(E_i,E_j)+D\Pi(W_1,E_i,E_j)+\Pi(E_i,D_{E_j}W_1)
\end{equation*}
because $D_{E_j}E_i(x)=0.$ Thus $DW_2=DV+\Pi(.,DW_1)+i(W_1)D\Pi$ for $x\in M.$ Substitution into (\ref{Eq26}) gives

\begin{equation}{\label{Eq29}}
\chi_0(\xi)=\frac{1}{2}(DW_2+D^{*}W_2)+K_{ol}(\xi)
\end{equation}

substitution of (\ref{Eq28}) into (\ref{Eq27}) give us

\begin{equation}{\label{Eq49}}
\varphi_0(\xi)=\frac{1}{2}Dw_1+\varphi_{ol}(\xi)
\end{equation}
where $K_{ol}(\xi)=-i(W_1)D\Pi+w_1c+w_2\Pi$ and $\varphi_{ol}(\xi)=-i(W_1)\Pi+\frac{W_2}{2}$
Assuming the material of the shell is homogeneous and isotropic we find in the literature (see for instance \cite{Ciarlet}) the stress-strain relations of the $3-$ dimensional shell on the middle surface $M$ and express energy as an integral over $M\times [\frac{-h}{2},\frac{h}{2}]$. Let $R>0$ be the smallest principal radius of curvature of the undeformed middle surface (see \cite{Ciarlet}, pg 166). As usual, for a thin shell it is assumed that $\frac{h}{R}<<1$ (see \cite{Ciarlet}, pg 18). Using this assumption, the following approximation of the strain energy of the shell is obtain (see \cite{Ciarlet}, pg 253)
\begin{equation}{\label{Eq30}}
\begin{aligned}
I(\xi)=&\alpha h \int_{M}\left\{\left.|\Upsilon(\xi)\right|^{2}+2\left|\varphi_{0}(\xi)\right|^{2}+\omega_{2}^{2}+\beta\left(tr(\Upsilon(\xi))+w_{2}\right)^{2}\right. \\
&\left.+\gamma\left[|\chi_0(\xi)|^{2}+\frac{\left|D w_{2}\right|^{2}}{2}+\beta\left(tr(\chi_0(\xi))\right)^{2}\right]\right\} d M
\end{aligned}
\end{equation}
where $\alpha=\frac{E}{1+\mu}$, $\beta=\frac{\mu}{1-2\mu}$ and $\gamma=\frac{h^2}{12}$. Here, $E$ denotes Young modules and $\mu$ is Poisson ratio $(0<\mu<\frac{1}{2})$.

The above expression (\ref{Eq30}) of $I(\xi)$ allows us to consider the following symetric, bilinear form $\tilde{B_0} $ associated to the strain energy, definied on the space $Z=[H^1(M,\wedge)]^2\times [H^1(M)]^2:$
\begin{equation}{\label{Eq31}}
\tilde{B_0}(\xi,\theta)=\frac{\alpha h}{2}\int_M B_0(\xi,\theta)dM
\end{equation}
where $\xi=(W_1,W_2,w_1,w_2)\in Z$, $\theta=(\theta_1,\theta_2,u_1,u_2)\in Z$ and 
\begin{equation}{\label{Eq32}}
\begin{aligned}
&B_{0}(\xi, \theta)=2\langle \Upsilon(\xi),\Upsilon(\theta)\rangle \\
&\left.+4<\varphi_{0}(\xi), \varphi_{0}(\theta)\right\rangle+2 \omega_{2} u_{2}+ \\
&+2 \beta \cdot\left(\operatorname{tr}( \Upsilon(\xi))+\omega_{2}\right)\left(t_{2} \Upsilon(\theta)+u_{2}\right) \\
&\left.\left.+2 \gamma<\chi_0(\xi), \chi_0(\theta)\right)\right\rangle+\gamma\left\langle D \omega_{2}, D u_{2}\right\rangle \\
&+2 \gamma \beta \operatorname{tr}(\chi_0(\xi)).\operatorname{tr}(\chi_0(\theta))
\end{aligned}
\end{equation}

In order to obtain Green identify we consider the Hodge-Laplace type operator $\Delta_{\beta}$ given by:

\begin{equation}{\label{Eq33}}
\Delta_{\beta}=-[\delta d+2(1+\beta) d \delta]
\end{equation}
where $\beta=\frac{\mu}{1-2\mu}$, $d$ is the exterior derivative and $\delta$ is its formal adjoint. The operator $\Delta_{\beta}$ acts taking a $p$-form to another $p$-form. In our case, we will need only $\Delta_{\beta}$ operating on $1$-forms.
In \cite{Yao}(theorem 5.1) the following result was prove:
Let us consider the bilinear form $\tilde{B_0}(.,.)$ given by(\ref{Eq31}). Then, for any $\xi=(W_1,W_2,w_1,w_2)\in Z$ and $\theta=(\theta_1,\theta_2,u_1,u_2)\in Z$, the identity:
\begin{equation}{\label{Eq34}}
\widetilde{B}_{0}(\xi, \theta)=\frac{\alpha h}{2}\left\langle A_{0} \xi, \theta\right\rangle L^{2}+\frac{\alpha h}{2} \int_{\Gamma=\partial M} \partial(A_0\xi,\theta) d \Gamma
\end{equation}
holds, where $L^2=[L^2(M,\wedge)]^2\times [L^2(M)]^2$ and
\begin{equation}\label{Eq35}
\begin{aligned}
A_{0}(\xi)=-&\left(\Delta \omega_{1}+F_{1}(\xi), \gamma \Delta_{\beta} W_{2}+F_{2}(\xi),\right.\\
&\left.\Delta w_{1}+f_{1}(\xi), \gamma \Delta \omega_{2}+f_{2}(\xi)\right) \\
\partial\left(A_{0} \xi, \theta\right) &=\left\langle c_{1}(\xi), \theta_{1}\right\rangle+\gamma\left\langle\operatorname{c}_{2}(\xi), \theta_{2}\right\rangle \\
&+2\left\langle\varphi_{0}(\xi) ; \eta\right\rangle u_1+\gamma \frac{\partial \omega_{2}}{\partial \eta} u_{2}
\end{aligned}
\end{equation}
with 
\begin{equation}{\label{Eq36}}
\begin{aligned}
&c_{1}(\xi)=2 i(\eta) \Upsilon(\xi)+2 \beta\left(\operatorname{tr} \Upsilon(\xi)+w_{2}\right) \eta \\
&c_{2}(\xi)=2 \gamma i(\eta) \chi_{0}(\xi)+2 \beta(\operatorname{tr}(\chi_0(\xi)) \eta
\end{aligned}
\end{equation}
here $\eta$ denotes the exterior normal vector along the curve $\Gamma=\partial M,$ $\Delta$ is the usual Laplace-Beltrami operator on the Riemannian manifold $M$ and $F_j(\xi)$ and $f_j(\xi)$ are first order terms, i.e, of order $\leq 1$ for $j=1\,\,\mbox{or}\,\, 2$ 

The above description was shown in \cite{Yao}. In fact, the variable $\xi=(W_1,W_2,w_1,w_2)$ satisfies the following system
$$
\begin{cases}W_{1}^{\prime \prime}-\Delta W_{\beta}+F_{1}(\xi)=0 & \text { on } M \times(0, \infty) \\ W_{2}^{\prime \prime}-\Delta W_{2}+F_{2}(\xi)=0 & \text { on } M \times(0,+\infty) \\ w_{1}^{\prime \prime}-\Delta w_{1}+f_{1}(\xi)=0 & \text { on } M \times(0,+\infty) \\ w_{2}^{\prime \prime}-\Delta w_{2}+f_{2}(\xi)=0 & \text { on } M \times(0,+\infty)\end{cases}
$$
with

$$
\left\{\begin{array}{l}
\xi=0 \quad \text { on } \Gamma(M) \times(0,+\infty) \\
\xi(0)=\xi_{0}, \xi_{t}(0)=\xi \quad \text{on}\quad \mbox{M}
\end{array}\right. 
$$ 

Here $\Gamma(M)$ denotes the boundary of $M$. Let us consider the following spaces
$$
\begin{aligned}
&H_{\Gamma}^{1}(M)=\left\{u \in H^{1}(M), u \equiv 0 \text { on } \Gamma=\Gamma(M)\right\} \\
&H_{\Gamma}^{1}(M, \Lambda)=\left\{z \in H^{1}(M, \Lambda), z \equiv 0 \text { on } \Gamma=\Gamma(M)\right\}
\end{aligned}
$$
and 
$$
X_{\Gamma}=\left[H_{\Gamma}^{1}(M, N)\right]^{2} \times\left[H_{\Gamma}^{\prime}(M)\right]^{2}
$$
Next, we can prove that the symetric bilinear form $\tilde{B}(\xi, \theta)$ defined in (\ref{Eq34}) for any $\bar{\xi}, \theta \in Z$ is coercive, that is, there exits a positive constant $c_3$ such that 
\begin{equation}{\label{Eq37}}
\widetilde{B}(\xi, \xi) \geqslant c_{3}\|\xi\|_{X_{\Gamma}}^{2} \text { for any } \xi \epsilon X_{\Gamma}
\end{equation}
In fact, substitution of (\ref{Eq25}), (\ref{Eq29}) and (\ref{Eq49}) into (\ref{Eq32}) followed by integration over $M$ give us 
\begin{equation}{\label{Eq38}}
\tilde{B}_{0}(\xi, \xi)+K_{2}\|\xi\|_{L^{2}}^{2} \geqslant K_{1}\|\xi\|_{H_{\Gamma}^{1}}^{2}(M)
\end{equation}
for some positive constants $K_1$ and $K_2$.

In order to use (\ref{Eq38}) to obtain (\ref{Eq37}) we can use the following uniqueness results: Let $\xi=\left(W_{1}, W_{2}, w_{1}, w_{2}\right)$ belonging to $X_{\Gamma}$ such that 
$\Upsilon(\xi)=0, \chi_{0}(\xi)=0, \varphi_{0}(\xi)=0$ and $w_2=0$. Then $\xi =0$ for all $x\in M$. Using this uniqueness result together with the method called compactness-uniqueness we can ``adsorb" the term $K_2\|\xi\|_{L^{2}}^{2}$ into the term of the right hand side of (\ref{Eq38}) to conclude the validity of (\ref{Eq37}). The expression (\ref{Eq34}) for the bilinear form $\widetilde{B}_{0}(\xi, \theta)$ is known. Using the above discusiion we deduce that the variational problem associated to the bilinear form $\widetilde{B}_{0}$ it is equivalent to the following boundary value problem: To find $\xi=\left(\boldsymbol{W}_{1}, \mathbf{W}_{2}, \mathbf{w}_{1}, \mathbf{w}_{2}\right)$ such that 
\begin{equation}{\label{Eq39}}
\left\{\begin{array}{l}
\frac{\alpha h}{2} A_{0} \xi=\tilde{F} \\
\left.W_{1}\right|_{\Gamma}=\left.W_{2}\right|_{\Gamma}=0,\left.w_{1}\right|_{\Gamma}=\left.w_{2}\right|_{\Gamma}=0
\end{array}\right.
\end{equation}
for a given $\tilde{F} \in \mathbf{L}^{2}=\left[L^{2}(M, \Lambda)\right]^{2} \times\left[L^{2}(M)\right]^{2}$

\subsection{Equation of motion}
In this Section we will consider the equations of motion of Naghdi model. We assume that there are no external loads on the shell and the shell is clamped along $\Gamma=\Gamma(M)$. In this situation $\tilde{\xi}=\tilde{\xi}(x,t),$  $x \in M$ and $t$ is time. In onder to include the kinetic energy to our problem it is convevient to consider the change of variables $t \rightarrow t c_{1}^{-1}$ where $c_{1}^{2}=\frac{2}{\alpha}$ with $\alpha$ as in (\ref{Eq30}). Also, denoting by
$$
R=\left[\begin{array}{llll}
1 & 0 & 0 & 0 \\
0 & \gamma & 0 & 0 \\
0 & 0 & 1 & 0 \\
0 & 0 & 0 & \gamma
\end{array}\right]
$$
and $\xi=R^{1 / 2} \cdot \widetilde{\xi}$ we write the equations of evolution of a Naghdi shell for $\xi=\left(\boldsymbol{W}_{1}, \mathbf{W}_{2} \mathbf{w}_{1}, \mathbf{w}_{2}\right)$ as 
\begin{equation}{\label{40}}
\left\{\begin{array}{l}
\xi+A \xi=0 \text { on } M \times \mathbb{R}^{+} \\
\xi=0 \text { on } \Gamma(M) \times \mathbb{R}^{+} \\
\xi(x, 0)=\xi_{0}(x), \xi_{t}(x, 0)=\xi_{1}(x) \text { on } M
\end{array}\right.
\end{equation}
where $	\mathbf{A}= R^{-1 / 2} \mathbf{A}_{0} R^{-1 / 2}$. The bilinear form $J$ associated to operator $\mathbf{A}$ is given by 
\begin{equation}{\label{42}}
\begin{aligned}
\mathbf{J}(\xi, u)&= 2\langle\Upsilon(\xi), \Upsilon(u)\rangle+4\left\langle\boldsymbol{\varphi}_{0}(\xi), \boldsymbol{\varphi}_{0}(u)\right\rangle \gamma \\
&+2 \beta\left[Tr (\Upsilon(\xi))+\frac{w_{2}}{\sqrt{\gamma}}\right]\left[Tr( \Upsilon(u))+\frac{u_{2}}{\sqrt{\gamma}}\right] \\
&+2 \beta Tr( \chi_0(\xi)) Tr( \chi_0(u))+2\left\langle\chi_0(\xi), \chi_0(u)\right\rangle \\
&+\left(D w_{2}, D u_{2}\right\rangle+\frac{2}{\gamma} w_{2} u_{2}
\end{aligned}
\end{equation}
for all $\xi=\left(W_{1}, W_{2}, w_{1}, w_{2}\right)$ and $u=\left(U_{1}, U_{2}, u_{1}, u_{2}\right)$  with  $\xi \text { and } u \in Z$. Also $\Upsilon, \chi_0$ and $\varphi_0$ are as in (\ref{Eq25}), (\ref{Eq26}) and (\ref{Eq27}) respectively. The bilinear form associated with the operator $A \text { is } \mathbf{J}(\xi, u)$ and the coresponding Green formula would be 
\begin{equation}{\label{43}}
\widetilde{J}(\xi, u)=\langle\mathbf{A} \xi, u\rangle_{Y}+\int_{\Gamma=\partial M} \partial\left(\mathbf{A}\xi, u\right) d r
\end{equation}
where $\tilde{J}(\xi, u)=\int_{M} \mathbf{J}(\xi, u) d M$. Now, we consider a localized perturbation model (\ref{40}): To find $\xi=\left(W_{1}, W_{2}, w_{1}, w_{2}\right) \in Z$ satisfying 
\begin{equation}{\label{44}}
\left\{\begin{array}{l}
\xi_{t t}+A \xi+a(x) \,\, \xi_{t}=0 \quad \text { on } M \times (0,+\infty) \\
\xi=0 \text { on } \Gamma(M) \times(0,+\infty) \\
\xi(x, 0)=\xi_{0}(x), \xi_{t}(x, 0)=\xi_{1}(x) \text { on } M
\end{array}\right.
\end{equation}
where $a(x)$ is a real valued function defined for all $x\in M$ such that has support in a small interior region of $M$. Well posedness of problem (\ref{44}) it follows using standar tools, for example, the semigroup theory \cite{Pazy} and omitt the proof here. Before we present a proof of the uniform stabilization of the total energy of model (\ref{44}) we need some preliminaries

Let us denote by $T^{2}(M)$ the set of all tensor fields on $M$ of rank $2$. We define the bilinear form $b(\cdot,\cdot):\quad T^{2}(M) \times T^{2}(M) \mapsto \mathbb{R}$ gives by
\begin{equation}{\label{45}}
b\left(T_{1}, T_{2}\right)=\left\langle T_{1}, T_{2}\right\rangle+\beta(trace(T_1))(trace(T_2))
\end{equation}
for any $T_1,T_2\in T^{2}(M)$. Here, for any $T\in T^{2}(M)$, the trace of $T$ at $x\in M$ is given by $trace(T)= \sum_{i=1}^{2}T(e_i,e_i)$ where $\{e_1,e_2\}$ is an ortonormal basis of $T_x M$. For any $W \in H^{1}(M, \Lambda	)$ we define 
\begin{equation}{\label{46}}
\widetilde{S}(W)=\frac{1}{2}\left(D W+D^{*} W\right)
\end{equation}
It is known that there exist a positive constant $\lambda$ such that
\begin{equation}{\label{47}}
2\|\tilde{S}(W)\|_{L^{2}\left(M, T^{2}\right)} =\left\|D W+D^{*} W\right\|_{{L^{2}\left(M, T^{2}\right)}}{\geq \lambda\|W\|_{H^1(M, \Lambda)}}
\end{equation}
for any $w\in H^1(M, \Lambda)$. See lemma $(4.5)$ in \cite{Yao}. We claim that there exist $\lambda_0 >0$ such that 
\begin{equation}{\label{48}}
\lambda_0 \int_M\left[b(\widetilde{S}(w), \tilde{S}(w))+|W|^2\right] d M \geqslant\|D w\|_{_{L^{2}\left(M,\Lambda\right)}}^{2}
\end{equation}
holds for any $W \in H_{\Gamma}^1(M,\Lambda)$. In fact,
$$
b (\tilde{S}(w), \tilde{S}(w))=\frac{1}{4}\left|D w+D^{*} w\right|^2+\beta\left(Trace \tilde{S}(w)\right)^2
$$
consequently,
\begin{equation}{\label{49}}
b(\widetilde{S}(w), \widetilde{S}(w))+|w|^2 \geq \frac{1}{4}\left|D W+D^* W\right|^2
\end{equation}
integration of (\ref{49}) over $M$ and using (\ref{47}) we obtain,
$$
\begin{aligned}
\int_m\left[b(\tilde{S}(w), \widetilde{S}(W))+\mid W i^2\right] d M & \geqslant \frac{1}{4}\left\|D W+D^* w\right\|_{L^2(M,T^2)}^{2} \\
& \geqslant \frac{\lambda}{4}\|w\|_{H^1(M, \Lambda)}^{2} \\
& \geqslant \frac{\lambda}{4}\|D W\|_{L^2(M, \Lambda)}^{2}
\end{aligned}
$$
wich prove our claim.
Next, was will use the tedmiage of multiplies to stain appropiatte identitis and inequalities. wet us assume that given $V \in \chi(M)$ there exists a function $v(x) \in C^{\infty}(M)$ such that
\begin{equation}{\label{50}}
D V(x)(X, X)=v(x)|X|^2
\end{equation}
for all $X \in T_x(M), x \in M$. Given $\xi=\left(W_1, W_2, w_1, w_2\right) \in Z$ we consider 
\begin{equation}{\label{51}}
m(\xi)=\left(D_V W_1, D_V W_2, V\left(w_1\right), V\left(w_2\right)\right)
\end{equation}
we recall that $D V(X, X)=\operatorname{DV}_X(X)=\left\langle\operatorname{DV}_X, X\right\rangle$, for all $X \in T_x(M), x \in M$.

Using our assumption (\ref{50}), we take the inner product of equation (\ref{44}) with $m(\xi)$ and integrate over $M$
\begin{equation}{\label{52}}
\begin{aligned}
&\left\langle\xi_{tt}, m(\xi)\right\rangle_{L^2(M,\Lambda)}+\langle A \xi, m(\xi)\rangle+\langle a(x) \xi_t, m(\xi)\rangle=0
\end{aligned}
\end{equation}
we can use identity (\ref{43}) to deduce from (\ref{52})
\begin{equation}{\label{53}}
\begin{aligned}
&\left\langle\xi_{tt}, m(\xi)\right\rangle_{L^2(M, \Lambda)}+ \tilde{J}(\xi, m(\xi))-\int_{\Gamma=\Gamma(M)} \partial(A \xi, m(\xi)) d \Gamma=-\left\langle a(x) \xi_t, m(\xi)\right\rangle.
\end{aligned}
\end{equation}
Using similar calculations as the ones given in Lemma $5.2$ in \cite{Yao} we deduce the identity
\begin{equation}{\label{54}}
\begin{aligned}
2 \widetilde{J}(\xi, m(\xi))=& \int_{\Gamma} J(\xi, \xi)\langle V, \eta\rangle d \Gamma 
-2 \int_M v J(\xi, \xi) d M+2 \int_M K(\xi, \xi) d M \\
&+l_0(\xi)
\end{aligned}
\end{equation}
where $\eta$ is the outside normal vector field along $\Gamma=\Gamma(M)$,
\begin{equation}{\label{55}}
\begin{aligned}
K(\xi, \xi) &=2 b(\tilde{S} ( W_1), G\left(V, D W_1\right))+2 b (\tilde{S}\left(W_2\right), G\left(V, D W_2\right))+\\
&+4 v\left|\Psi_0(\xi)\right|^2+v\left|D W_2\right|^2.
\end{aligned}
\end{equation}
Here $b(\cdot,\cdot)$ is as in (\ref{45}) and $G$ is a map defined by,
\begin{equation}{\label{56}}
\begin{aligned}
&G: \chi(M) \times T^2(M) \longmapsto T^2(M) \\
&G(W, T)=\frac{1}{2}\left[T(\cdot, \nabla \cdot W)+T^*(\cdot, W \cdot W)\right]
\end{aligned}
\end{equation}
Finally, in (\ref{54}), $l_0(\xi)$ denotes lower order terms with respect to the energy in the sense that for any  $\epsilon>0$ there exists $C_{\epsilon}>0$ such that 
\begin{equation}{\label{57}}
\left|l_0(\xi)\right|^2 \leqslant \varepsilon d(\xi)+C_{\varepsilon} h(\xi) \text { for all } x \in M
\end{equation}
where $d(\xi)$ is the density energy $e(t)=\frac{1}{2} \int_M d(\xi) d M$ y $h(\xi)$ has partial derivatives only up to order $1$.

\section{The stabilization result}
Let $\xi$ be the displacement field of Naghdi shell. The total energy 	of the model is 
\begin{equation}{\label{58}}
E(t)=\frac{1}{2}\left\|\xi_t\right\|_{L^2(M)}^2+\frac{1}{2} \tilde{J}(\xi, \xi) 
\end{equation}
where $\tilde{J}$ is given as in(\ref{43}). In order to study the stabilization result for the solutions of model (\ref{44}) we will consider the following assumption on $a(x)$

{\bf Assumption 1} a:M $\longmapsto \mathbb{R}^{+}$is a real valued function nonnegative and with support in a small interior region of $M$.
\begin{theorem}{\label{Th10}}
Consider the solution of problem (\ref{44}), $\text { (with } \Gamma_0=\Gamma \text { ) }$ and initial data $\xi_0 \in Z=\left[\mathbf{H}^{1}(M, 1)\right]^2 \times\left[\mathbf{H}^{1}(M)\right]^2$, $\left.\xi_1 \in Y=\left[L^2(M, \Lambda)\right]^2 \times\left[L^2(M)\right]\right]^2$. Assume that condition (\ref{50}) holds. Then, the identity:
\begin{equation}{\label{80}}
\begin{aligned}
& \int_0^T \int_{\Gamma}\left[2 \partial(A \xi, m(\xi))+\left(\left|\xi_t\right|^2-J(\xi, \xi)\right)\langle V, \eta\rangle\right] d r d t \\
& \left.\quad=2\left(\xi_t, m(\xi)\right)_{L^2}\right]_0^T+2 \int_0^T \int_M v\left[|\xi_t|^2-J(\xi \xi)\right] d M d t \\
& +2 \int_0^T \int_M^T k(\xi, \xi) d M d t-\int_0^T\left(a(x) \xi_t, 2 m(\xi)\right) d t \\
& \quad+l_0(\xi)
\end{aligned}
\end{equation}
holds.
\end{theorem}
holds. Here $m(\xi)$ is a in (\ref{51}) and $K$ given in (\ref{55}). Furthermore, if we consider $P: M \longrightarrow \mathbb{R}$ any smooth function, then, the identity
\begin{equation}{\label{801}}
\begin{aligned}
\int_0^T \int_{\Gamma} \partial(A \xi, p \xi) d \Gamma d t & =\int_0^T \int_M p\left[J(\xi, \xi)-|\xi|^2\right] dM d t \\
& -\int_0^T\left(a(x) \xi_t, p \xi\right)_{L^2} d t+l_0(\xi) \\
\end{aligned}
\end{equation}
holds.

{\bf Proof:}
As before we can use (\ref{52}) and (\ref{43}) to obtain from equation (\ref{44}) the identity
\begin{equation}{\label{81}}
\begin{aligned}
& \left\langle\xi_{tt}, 2 m(\xi)\right\rangle_{L^2(M, \lambda)}+\tilde{J}(\xi, 2 m(\xi)\rangle \\
& -\iint_{\Gamma} \partial(A \xi, 2 m(\xi)) d \Gamma=-\left\langle a(x) |\xi_t|, 2m(\xi)\right\rangle. \\
&
\end{aligned}
\end{equation}
Using (\ref{50}), (\ref{51}) we obtain after taking inner product, integrate over $M$ and use Green formula
\begin{equation}{\label{82}}
-2\left\langle\xi_t, m\left(\xi_t\right)\right\rangle_{L^2}=2 \int_M v\left|\xi_t\right|^2 d M-\int_{\Gamma=\partial M}\left|\xi_t\right|^2\langle V, \eta\rangle d \Gamma
\end{equation}
Substitution of (\ref{82}) into (\ref{81}) gives us
\begin{equation}{\label{83}}
\begin{aligned}
\left\langle\xi_{t t}, 2 m(\xi)\right\rangle_{L^2} & =2 \frac{\partial}{\partial t}\left\langle\xi_t, m(\xi)\right\rangle_{L^2}+2 \int_M v\left|\xi_t\right|^2 d M \\
& -\int_{\Gamma=\partial M}\left|\xi_t\right|^2\langle V, \eta\rangle d \Gamma
\end{aligned}
\end{equation} 
Using Green formula for operator $A$ and relation (\ref{54}) we deduce
\begin{equation}{\label{84}}
\begin{aligned}
\langle A \xi, 2 m(\xi)\rangle_{L^2} & =\int_{\Gamma=\partial M}[\boldsymbol{J}(\xi, \xi)\langle V, \eta\rangle -2 \partial(\boldsymbol{A} \xi, m(\xi))] d \Gamma \\
&-2 \int_M v \boldsymbol{J}(\xi, \xi) d M+2\int_{M}K(\xi, \xi) d M +l_0(\xi)
\end{aligned}
\end{equation}
Using identities (\ref{83}), (\ref{84}) follow (\ref{801}). We need some geometric hypothesis in orde to obtain the desired stabilization result in case of localized internal dissipation.
\begin{definition}{\label{Camp_Esca_NagEv}}
Let $V$ a vector field on $M$ that is $V \in X(M)$. We say that $V$ is an escape vector field for the Naghdi shell if the following conditions are satisfy:
\begin{itemize}
	\item[a)] There exists a function $v$ on $M$ such that 
	$$
	D V(X, X)=v(x)|X|^2
	$$
	for all $X \in T_x(M), x \in M$
	\item[b)] Let $\varepsilon(x)$ denote the ``volume element of the middle surface M", consider 
	$$f(x)=\frac{\langle D V, \varepsilon\rangle}{2} \quad \text{for}\,\,\, x \in M$$
    The functions $v(x)$ and $l(x)$ are assume to satisfy the inequality 
    $$2 \min_{x \in M} r(x)>\lambda_0(1+2 \beta) \max_{x \in M}|\ell(x)|$$
    where  $\lambda_0 \geqslant 1$ satisfies (\ref{48}) with $\lambda_0^{-1}=\frac{c}{4}$ and $\beta=\frac{\mu}{1-2 \mu}$
\end{itemize}
\end{definition}
\section{Some comments on scape vector fields:}
\begin{itemize}
	\item[1.] It is well known that on a $2$-dimensional middle surface $M$ there always exists a vector freed $V$ satisfying assumption a). To obtain assumption b) may be the difficult part. It is known that a necessary condition to hold condition b) is that there is no closed geodesics inside the middle surface $M$.
	\item[2.] Condition b) says in a sense that function $l(x)$ is related to the symmetry of the covariant differential DV. In fact, if $D V$ is symmetric then $l(x)=0$ for all $x \in M$.
	\item[3.] In our case $M$ is an oriented Riemsannian manifold of $\operatorname{din} M=2$. Let $\left\{e_1, e_2\right\} \in T_x(M)$ be linearly independent and let $\varepsilon$ be a differential form of degree 2 defined as
\begin{align*}
\varepsilon\left(e_1, e_2\right)(x) &= \pm \sqrt{\operatorname{det}\left(\left\langle e_i e_j\right\rangle\right)}\\
  &= \text{orient.vol}\{e_1, e_2\}
\end{align*}
$x\in M$. The oriented volume is affected by the sign $+$ or $-$ depending on wheter or not a basis $\{ e_1, e_2\}$ belongs to the orientation on $M$. $\varepsilon=\varepsilon(x)$ is called the volume element of $M$. Some texts define the volume element as a $2-$ form on $M$ if for any frame field $\{e_1, e_2\}$, $|\varepsilon(e_1,e_2)|=1$. In chapter $4$ of \cite{Yao} several examples are presented where we can assure the construction of escape vector fields for a shallow shells.
\end{itemize}
Next, we define the motion of scape region a piece $\overline{M}$ of the middle surface $M$ wich will be convenient in order to obtain the desired stabilization result. We are interested in case where $\overline{M}$ is not the whole $M$ and maybe small as possible (in some sense). The notion was used by several authors in related work (see \cite{Yao}, \cite{Petersen}, and references therein)

\begin{definition}{\label{Reg_Esca_Nagh}}
A region $G\subset\Omega$ is called a scape region for the Naghdi shell, if
	\begin{enumerate}
		\item[$1)$]There is a finite number of sub-regions $\left\{ \Omega_i \right\}_{i=1}^{J}$, with boundary $\Gamma_i$, $J$ a natural positive, such that
		\[\Omega_i\cap\Omega_j=\phi\quad\mbox{para todo}\quad 1\leq i<j\leq J.\]
		\item[$2)$]For each $\Omega_i$ there is a vector field $V_i$  and a function $v_i$ such that,
		\begin{eqnarray*}
		DV_i(X,X)&=& v_i(x)|X|^2\quad\mbox{on}\quad\Omega_i\\
		2\min_{x\in\Omega_i}v_i(x)&>&\lambda_0(1+2\beta)\max_{x\in\Omega_i}\frac{|l_i(x)|}{2},
		\end{eqnarray*}
		where $l_i(x)=\frac{\left< DV_i,E \right>}{2}$ for all $1\leq i\leq J$;
		\item[$3)$] \[G\supset\bar{\Omega}\cap N_{\epsilon}\left[\cup_{i=1}^{J}\Gamma_{i0}\cup\left( \Omega\setminus\cup_{i=1}^{J}\Omega_i \right) \right]  \]
		where $\epsilon>0$, small, and:
		\begin{eqnarray*}
		N_{\epsilon}(S)&=&\cup_{x\in S}\left\{ y\in\Omega/d_g(y,x)<\epsilon \right\}\quad{para}\quad S\subset M \\
		\Gamma_{i0}&=&\left\{ x\in\Gamma_i, \left< V_i(x),\nu(x) \right> >0 \right\},
		\end{eqnarray*}
		$\nu_i$ is the normal to $\Omega_i$ pointing out. 
	\end{enumerate}
\end{definition}
In general, there is no defined scape vector field over the entire median surface $\Omega$. However, such fields can be defined on small geodesic balls. Then, considering $\Omega=\cup_{n\in I\!\!N}B(x_n,\delta)$ with $x_n\in\Omega$ and $\delta>0$ small enough, an escape vector field can be defined in $B(x_n,\delta)$. Then $\mu(\Omega)=\lim_{k\rightarrow\infty}\sum_{n=1}^{k}\mu(B(x_n,\delta))$, where $\mu$ is the two-dimensional Lebesgue measure on the surface $\Omega$. So, given $\epsilon>0$, there is $N\in I\!\!N$ big enough that
\[\sum_{\infty}^{n=N+1}\mu(B(x_n,\delta)) <\epsilon.\]
Then considering $\Omega_i=B(x_i,\delta)$ con $1\leq i\leq N$, we have proved the following
\begin{theorem}{\label{Regi_Peque}}
For $\epsilon>0$ given, an escape region can be chosen $G\subset\bar{\Omega}$ such that \[\mu(G)<\epsilon\]
where $\mu(G)$ is the two-dimensional Lebesgue measure of $G$
\end{theorem}
Now we will give some examples.
\begin{example}
In the case of an escape vector field, $V$, defined on all $\Omega$, then in the definition (\ref{Reg_Esca_Nagh}) we have to $J=1$. By condition($3$) a region escape is supported in the boundary region $\Gamma_0$, where
\[\Gamma_0=\left\{ x\in\partial\Omega / \left< V(x),\nu(x) \right> >0\right\}.\]
\end{example}
That escape region was already used by many authors\cite{Cavalcanti}, \cite{Bernadou_2}, \cite{Lions}. The escape field considered, in the case
 $I\!\!R^n$, was $V=x-x_0$

%
%

\begin{example}
consider now \[\Omega=C=\left\{ x=(x_1,x_2,x_3)\in I\!\!R^3 /x_1^2+x_2^2=1,\quad -1\leq x_3\leq 1 \right\},\]
the cylinder limited in $I\!\!R^3$. it is known that is not possible to define a vanishing vector field over all $\Omega$. To construct an escape region, let $x_0=\left( x_{01},x_{02},x_{03} \right)\in C$, with $x_{03}=0$ and either $L_0$ the generating line containing $x_0$. Be $x_1\in C$ the antipode of $x_0$. Since the interior of the cut-locus of $x_1$ is $C\setminus	\setminus L_0$, there exists an escape vector field defined over $C\setminus L_0$. Thus, an escape region for $\Omega$ is supported in a neighborhood of the edge of $\Omega$ and the $L_0$. 
\end{example}

\section{Naghdi shell stabilization with internal dissipation}{\label{Sec_Decaimento_Naghi}}
To continue with the resolution of the problem of the exponential decay of energy, we need the following lemmas.
\begin{lemma}{\label{Decom_DV}}
Let $V\in\chi(\Omega)$ satisfying the first condition of the definition (\ref{Camp_Esca_NagEv}). So the tensor field $DV$ can be decomposed as \[DV=v(x)g +l(x)E\quad\mbox{para}\quad x\in\Omega\]
\end{lemma}
\begin{proof}
decomposing $DV$ in its symmetric and antisymmetric part, by
\begin{equation}{\label{Cho1}}
DV=\frac{1}{2}\left( DV+D^*V \right)+\frac{1}{2}\left( DV-D^*V \right)
\end{equation}
Given that $\frac{1}{2}\left( DV-D^*V \right)$ is a $2$-antisymmetric shape and $\Omega$ is $2$-dimension, there is a function $q$ such that
\begin{equation}{\label{cho2}}
\frac{1}{2}\left( DV-D^*V \right)=q(x)E\quad\mbox{para}\quad x\Omega.
\end{equation}
Because the dimension of the $2$-shapes defined over spaces $2$-dimensionais is $1$. substituting (\ref{cho2}) in the expression of $l(x)$, we have
\begin{eqnarray*}
l(x)&=&\frac{1}{2}\left< DV,E \right>=\frac{1}{2}\left< 2q(x)E+D^*V,E \right>\\
&=& \left< q(x)E,E \right> + \frac{1}{2}\left< D^*V,E \right>=2 q(x)+\frac{1}{2}\left<D^*V ,E \right>\\
&=& 2 q(x)-\frac{1}{2}\left<DV ,E \right>\\
&=& 2 q(x) - l(x),
\end{eqnarray*}
de dónde, 
\begin{equation}{\label{cho3}}
l(x)=q(x).
\end{equation} 
substituting (\ref{cho3}) in (\ref{cho2}) and using that $DV=vg$ we have the result.
\end{proof}

Another notation to fix is the following. Given $W\in\chi(\Omega)$ e $T\in T^2(\Omega)$, be $G(W,T)\in T^2(\Omega)$ given by
\begin{equation}{\label{G_NEv}}
G(W,T)=\frac{1}{2}\left[ T(.,\nabla_{.}W)+T^*(.,\nabla_{.}W) \right]
\end{equation}
Now we prove some necessary lemmas
\begin{lemma}
There exists a constant $c>0$ such that
\begin{equation}{\label{Korn}}
|| DW+D^*W ||_{L^2(\Omega,T^2)}\geq c ||W||_{H^1(\Omega,\Lambda)}\quad\forall W\in H_{\Gamma_0}^1(\Omega,\Lambda)
\end{equation}
\end{lemma}

We note that, for $W\in H_{\Gamma_0}^1(\Omega,\Lambda)$
\begin{eqnarray}
b(S(W),S(W))&=&\left< S(W),S(W) \right> + \left( \mbox{Traç}(S(W)) \right)^2\nonumber\\
&=&\frac{1}{4}|DW+D^*W|^2 + \left( \mbox{Traç}(S(W)) \right)^2\nonumber\\
\end{eqnarray}
Then,
\[b(S(W),S(W)) + |W|^2 \geq  \frac{1}{4}|DW+D^*W|^2.\]
From where, integrating in $\Omega$ and using the lemma (\ref{Korn})
\begin{eqnarray}{\label{corc_bNEv}}
\int_{\Omega}\left[b(S(W),S(W)) + |W|^2 \right]dx &=& \frac{1}{4}||DW+D^*W||^2_{L^2(\Omega,T^2)}\geq \frac{c}{4}||W||^2_{H^1(\Omega,\Lambda)}\nonumber\\
&\geq & \frac{c}{4}||DW||^2_{L^2(\Omega,\Lambda)}\nonumber\\
\lambda_0 \int_{\Omega}\left[ b(S(W),S(W)) + |W|^2 \right]dx &\geq & ||DW||^2_{L^2(\Omega,\Lambda)}
\end{eqnarray}
with $\lambda_0 = \frac{4}{c}$.

\begin{lemma}{\label{Des_b}}
Be $V$ an escape vector field for the Naghdi shell model and let $G(V,DW)\in T^2(\Omega)$ given by (\ref{G_NEv}), for $W\in H^1(\Omega, \Lambda)$. Then, 

\[\sigma_1\int_{\Omega}b\left( S(W),S(W) \right)dx\leq \int_{\Omega}b\left( S(W),G(V,DW)dx +Lo(\xi) \right). \]
where $\sigma_1=\min_{x\in\Omega}v(x)-(1+2\beta)\max_{x\in\Omega}\frac{|l(x)|}{2}$.
\end{lemma}

\begin{proof}
Remembering that, for $T_1,T_2\in T^2(\Omega)$, we have:
\[b(T_1,T_2)=\left< T_1,T_2 \right> + \beta\mbox{Traç}T_1\mbox{Traç}T_2.\]
Then,
\begin{equation}{\label{cho4}}
b(S(W),G(V,DW))=\left< S(W),G(V,DW) \right> + \mbox{Traç}(S(W))\mbox{Traç}(G(V,DW))
\end{equation}
Now let estimate each term of (\ref{cho4}). For this, given $x\in\Omega$, be $\{e_1,e_2\}$ an orthonormal basis of $T_x\Omega$ such that
\begin{equation}{\label{cho5}}
DW(e_1,e_2)+D^*W(e_1,e_2)=0\quad\mbox{em}\quad x
\end{equation}
This is possible since the order tensor $2$, $DW+D^*W$, it is symmetrical. It follows that,
\begin{equation}{\label{cho6}}
S(W)(e_1,e_2)=0.
\end{equation}
Considering $W_{ij}=DW(e_i,e_j)$, have
\begin{eqnarray}{\label{cho7}}
\mbox{Traç}(S(W)) &=& S(W)(e_1,e_2)+S(W)(e_1,e_2)\nonumber\\
&=& DW(e_1,e_2) + DW(e_2,e_1)\nonumber\\
&=& W_11+W_22
\end{eqnarray}
\begin{eqnarray}{\label{cho8}}
\mbox{Traç} G(V,DW) &=& G(V,DW)(e_1,e_2) + G(V,DW)(e_2,e_2)\nonumber\\
&=& DW(e_1,\nabla_{e_1} V) + DW(e_2,\nabla_{e_2} V)
\end{eqnarray}
Now, using the lemma (\ref{Decom_DV}), have
\begin{eqnarray}{\label{cho9}}
\nabla_{e_1}V &=& \left< \nabla_{e_1}V,e_1 \right> e_1 + \left< \nabla_{e_1}V,e_2 \right> e_2\nonumber\\
&=& DV(e_1,e_2)e_1 + DV(e_2,e_1)e_2\nonumber\\
&=& v(x)e_1 -l(x)e_2
\end{eqnarray}

\begin{eqnarray}{\label{cho10}}
\nabla_{e_2}V &=& \left< \nabla_{e_2}V,e_1 \right> e_1 + \left< \nabla_{e_2}V,e_2 \right> e_2\nonumber\\
&=& DV(e_1,e_2)e_1 + DV(e_2,e_2)e_2\nonumber\\
&=& l(x)e_1 +v(x)e_2
\end{eqnarray}
substituting (\ref{cho10}), (\ref{cho9}) en (\ref{cho8}), have
\begin{eqnarray*}
\mbox{Traç}G(V,DW) &=& DW(e_1, v(x)e_1-l(x)e_2) + DW(e_2, l(x)e_1+v(x)e_2)\\
&=&v(x)DW(e_1,e_1)-l(x)DW(e_1,e_2)+l(x)DW(e_2,e_1)+v(x)DW(e_2,e_2)
\end{eqnarray*}
que, por (\ref{cho5}), tenemos
\begin{eqnarray}{\label{cho11}}
\mbox{Traç}G(V,DW) &=& v(x)(W_{11}+W_{22})+2l(x)W_{21}\nonumber\\
&=&v(x)\mbox{Traç}(DW) + 2l(x)W_{21}
\end{eqnarray}
substituting (\ref{cho6}), (\ref{cho7}) and (\ref{cho11}) in (\ref{cho4}), obtain:
\begin{eqnarray*}
b(S(W),G(V,DW)) &=& \beta\mbox{Traç}DW\left( v(x)\mbox{Traç}DW + 2l(x)W_{21} \right)\\
&=&\beta v(x)\left( \mbox{Traç} DW\right)^2 + 2\beta l(x)\left( W_{11}+W_{22} \right)W_{21}\\
&=& v(x)b(S(W),S(W))+2\beta l(x)\left( W_{11}+W_{22} \right)W_{21}\\
&\geq & \min_{x\in\Omega}v(x)b(S(W),S(W))-(1+2\beta)\max_{x\in\Omega}\frac{|l(x)|}{2}|DW|^2 +Lo(\xi),
\end{eqnarray*}
Integrating this last equation into $\Omega$, have:
\begin{equation}{\label{cho12}}
\int_{\Omega}b(S(W),G(V,DW))dx\geq \min_{x\in\Omega}\int_{\Omega}b(S(W),S(W))dx -(1+2\beta)\max_{x\in\Omega}\frac{|l(x)|}{2}\int_{\Omega}|DW|^2dx
\end{equation}
Finally, using (\ref{corc_bNEv}) in (\ref{cho12}), have:
\begin{eqnarray*}
\int_{\Omega}b(S(W),G(V,DW))dx &\geq & \min_{x\in\Omega}\int_{\Omega}b(S(W),S(W))dx \\
&-&\lambda_0(1+2\beta)\max_{x\in\Omega}\frac{|l(x)|}{2}\int_{\Omega}b(S(W),S(W))dx + Lo(\xi)\\
&=& \sigma_1\int_{\Omega}b(S(W),S(W))dx + Lo(\xi)
\end{eqnarray*}
\end{proof}

\begin{lemma}{\label{BNa_I}}
\[2\mathsf{B}(\xi,m(\xi))=\int_{\Gamma}\mathbf{B}(\xi,\xi)\left< V,\nu \right>d\Gamma-2\int_{\Omega}v\mathbf{B}(\xi,\xi)+2\int_{\Omega}e(\xi,\xi)dx +Lo(\xi)\]
\end{lemma}
where $\mathsf{B}$ is the bilinear form given in (\ref{Eq34}) and
\[e(\xi,\xi)=2b(S(W_1),G(V,DW_1))+2b(S(W_2),G(V,DW_2))+4v|\varphi(\xi)|^2+v|Dw_2|^2\]
\begin{proof}
by the formula (\ref{45}) we have to estimate $\Upsilon(\xi),\chi(m(\xi)),\varphi(m(\xi))$ e $\left< Dw_2,D(V(w_2)) \right>$. We start with the first term,

\begin{equation}{\label{Upsimxi}}
\Upsilon(m(\xi))=\frac{1}{2}\left[ D(\nabla_V W)+D^*(\nabla_V W) \right]+V(w_1)\Pi
\end{equation}
We will use the Bochner technique. Be $x\in\Omega$ e $\left\{ E_i \right\}_{i=1}^{2}$ a normal referential in $x$, we have:
\begin{eqnarray*}
D(\nabla_V W_1)(E_i,E_j) &=& E_j\left( \left< \nabla_{V}W_1, E_i \right> \right)\\
&=& E_j\left( DW_1(E_i,V) \right)=D^2W_1\left( E_i,V,E_j \right)+ DW_1(E_i,\nabla_{E_j}V)\\
&=& D^2W_1\left( E_i,E_j,V \right)+\left< R_{VE_j}W_1,E_i \right> + DW_1\left( E_i,\nabla_{E_j}V \right)\\
&=& \nabla_V DW_1(E_i,E_j)+R(V,E_j,W_1,E_i)+DW(E_i,\nabla_{E_j}V)
\end{eqnarray*}
Where from,
\begin{equation}{\label{chofi3}}
D(\nabla_V W_1)=\nabla_{V}DW_1 + R(V,.,W_1,.)+DW_1(.,\nabla_{.}V)
\end{equation}
Similarly,
\begin{equation}{\label{chofi4}}
D^*(\nabla_V W_1)=\nabla_{V}D^*W_1 + R(V,.,W_1,.)+DW_1(.,\nabla_{.}V)
\end{equation}
and,
\begin{eqnarray}
V(w_1\Pi)&=&V(w_1)\Pi +w_1\nabla_V \Pi\nonumber
\end{eqnarray}
So
\begin{equation}{\label{chofi5}}
V(w_1)\Pi = V(w_1\Pi)-w_1\nabla_V \Pi
\end{equation}

substituting (\ref{chofi5}), (\ref{chofi4}), (\ref{chofi3}) in (\ref{Upsimxi}), have
\begin{eqnarray}{\label{UpsmxiF}}
\Upsilon(m(\xi))&=&\frac{1}{2}\left\{ \nabla_V \left( DW_1+D^*W_1 \right)+DW_1(.,\nabla_{.}V)+ DW_1(\nabla_{.}V,.) \right\}\nonumber\\
&+&R(V,.,W_1,.)+V(w_1\Pi)-w_1\nabla_V \Pi\nonumber\\
&=&\nabla_{V}\Upsilon(\xi) + G(V,DW_1)+Lo(\xi)
\end{eqnarray}
where, $Lo(\xi)=R(V,.,W_1,.)-w_1\nabla_V\Pi$. Continuing with the following terms,
\begin{equation}{\label{chimxi}}
\chi(m(\xi))=\frac{1}{2}\left( D(\nabla_V W_2)+ D^*(\nabla_V W_2) \right) + V(w_2)\Pi -\sqrt{\gamma}\left( i(\nabla_V W_1)D\Pi-V(w_1)c \right)
\end{equation}
now estimating the terms of (\ref{chimxi}),
\begin{eqnarray*}
\nabla_V \left( i(W_1)D\Pi \right)(E_i,E_j)&=&D\left( i(W_1)D\Pi \right)(E_i,E_j,V)=V\left( i(W_1)D\Pi(E_i,E_j) \right)\\
&=& V\left( D\Pi(W_1,E_i,E_j) \right)\\
&=& D(D\Pi)(W_1,E_i,E_j,V)+D\Pi(\nabla_V W_1,E_i,E_j)\\
&=& i(W_1)\nabla_V D\Pi(E_i,E_j) + i(\nabla_V W_1)D\Pi(E_i,E_j)
\end{eqnarray*}
Therefore,
\begin{equation}{\label{chofi6}}
i(\nabla_V W_1)D\Pi = \nabla_V \left( i(W_1)D\Pi \right)-i(W_1)\nabla_V D\Pi
\end{equation}
substituting (\ref{chofi3}), (\ref{chofi4}), (\ref{chofi5}) y (\ref{chofi6}) in (\ref{chimxi}), have
\begin{eqnarray}{\label{chimxiF}}
\chi(m(\xi)) &=& \frac{1}{2}\left\{ \nabla_V\left( DW_2+D^*W_2 \right)+DW_2(.,\nabla_{.}V)+DW_2(\nabla_{.}V,.) \right\}\nonumber\\
&+& R(V,.,W_2,.)+V(w_2\Pi)-w_2\nabla_V\Pi - \sqrt{\gamma}\nabla_V (i(W_1)D\Pi)-\sqrt{\gamma}i(W_1)\nabla_V D\Pi\nonumber\\
&-& \sqrt{\gamma}V(w_1c)-\sqrt{\gamma}w_1\nabla_V c\nonumber\\
&=& \nabla_V \chi(\xi)+G(V,DW_2)+Lo(\xi)
\end{eqnarray}
continuing with $\varphi(\xi)$,
\begin{equation}{\label{varpmxi}}
\varphi(m(\xi))=\frac{1}{2}D(V(w_1))-i(\nabla_V W_1)\Pi + \frac{1}{\sqrt{\gamma}}\nabla_V W_2
\end{equation}
Estimating the terms of (\ref{varpmxi}),
\begin{eqnarray*}
\left< D(V(w_1)),E_i \right> &=& E_i(V(w_1))=E_i(\left< Dw_1,V \right>)\\
&=& \left< \nabla_{E_i}Dw_1, V \right>+\left< Dw_1,\nabla_{E_i} V \right>\\
&=& \left< \nabla_{V}Dw_1, E_i \right> + Dw_1\left( \nabla_{E_i}V \right)
\end{eqnarray*}
Therefore,
\begin{equation}{\label{chofi8}}
D(V(w_1))=\nabla_V Dw_1 + Dw_1\left( \nabla_{.}V \right)
\end{equation}
continuing,
\begin{eqnarray*}
\left< \nabla_V\left( i(W_1)\Pi, E_i \right) \right> &=& D(i(W_1)\Pi)(E_i,V)=V(i(W_1)\Pi(E_i))=V(\Pi(W_1,E_i))\\
&=& D\Pi(W_1,E_i,V)+\Pi(\nabla_V W_1,E_i)\\
&=& \nabla_V\Pi(W_1,E_i) + \left< i(\nabla_V W_1)\Pi, E_i \right>\\
&=& \left< i(W_1)\nabla_V\Pi, E_i \right>+ \left< i(\nabla_V W_1)\Pi, E_i \right>
\end{eqnarray*}
Therefore,
\begin{equation}{\label{chofi9}}
i(\nabla_V W_1)\Pi = \nabla_V\left( i(W_1)\Pi \right) - i(W_1)\nabla_V\Pi
\end{equation}
substituting (\ref{chofi8}), (\ref{chofi9}) in (\ref{varpmxi}), have,
\begin{eqnarray}{\label{varpmxiF}}
\varphi(m(\xi))&=&\frac{1}{2}\left( \nabla_V Dw_1 + Dw_1\left( \nabla_{.}V \right) \right) -\nabla_V \left( i(W_1)\Pi \right) + i(W_1)\nabla_V \Pi+ \frac{1}{\sqrt{\gamma}}\nabla_V W_2 \nonumber\\
&=& \nabla_V \varphi(\xi) + \varphi(\xi)\left( \nabla_{.}V \right) +Lo(\xi)
\end{eqnarray}
Now writing the equation (\ref{Eq35}), with $\eta=m(\xi)$, have
\begin{eqnarray}{\label{BNEvProv}}
\mathbf{B}(\xi,m(\xi))&=& 2\left< \Upsilon(\xi),\Upsilon(m(\xi)) \right> + 4\left< \varphi(\xi),\varphi(m(\xi)) \right> + 2\beta\left( \mbox{Traç}(\Upsilon(\xi)+\frac{1}{\sqrt{\gamma}})w_2 \right)\nonumber\\
& & \left( \mbox{Traç}(\Upsilon(m(\xi))+\frac{1}{\sqrt{\gamma}})V(w_2) \right)
+ 2\left< \chi(\xi),\chi(m(\xi)) \right> + 2\beta\mbox{Traç}(\chi(\xi))\mbox{Traç}(\chi(m(\xi)))\nonumber\\
&+& \left< Dw_2, D(V(w_2)) \right> + \frac{1}{\gamma}w_2V(w_2)\nonumber\\
\end{eqnarray}
using (\ref{UpsmxiF}), (\ref{chimxiF}) and (\ref{varpmxiF}) Let estimate each term of (\ref{BNEvProv}).
\begin{eqnarray*}
2\left< \varphi(\xi),\varphi(m(\xi)) \right> &=& 2\left< \varphi(\xi),\nabla_V\varphi(\xi)+\varphi(\xi)\left( \nabla_{.}V \right) + Lo(\xi) \right>\\
&=& 2\left< \varphi(\xi),\nabla_V\varphi(\xi)\right> + 2\left< \varphi(\xi), \varphi(\xi)\left( \nabla_{.}V \right) \right> +Lo(\xi)\\
&=& V(|\varphi(\xi)|^2) + 2\left< \varphi(\xi),\nabla_{\varphi(\xi)}V \right> + Lo(\xi)\\
&=& V(|\varphi(\xi)|^2) + 2DV(\varphi(\xi), \varphi(\xi)) +Lo(\xi).
\end{eqnarray*}
From where, using (\ref{Camp_Esca_NagEv}), we have
\begin{equation}{\label{varpivarmxi}}
2\left< \varphi(\xi),\varphi(m(\xi)) \right> = V(|\varphi(\xi)|^2) + 2v|\varphi(\xi)|^2 + Lo(\xi)
\end{equation}
Continuing with the terms of (\ref{BNEvProv}),
\begin{eqnarray}{\label{DW2DmW2}}
2\left< Dw_2, D(V(w_2)) \right> &=& 2\left< Dw_2, \nabla_V Dw_2 + Dw_2\left( \nabla_{.}V \right) \right>\nonumber\\
&=& 2\left< Dw_2, \nabla_V Dw_2 \right> + 2\left< Dw_2, \nabla_{Dw_2}V \right>\nonumber\\
&=& V(|Dw_2|^2) + 2v|Dw_2|^2
\end{eqnarray}
and,
\begin{eqnarray}{\label{upsxiupmxi}}
2\left< \Upsilon(\xi),\Upsilon(m(\xi)) \right> &=& \left< \Upsilon(\xi), \nabla_V \Upsilon(\xi)+G(V,DW_1)+Lo(\xi) \right>\nonumber\\
&=& V(|\Upsilon(\xi)|^2) + 2\left< \Upsilon(\xi),G(V,DW_1) \right> +Lo(\xi)
\end{eqnarray}
Continuing with the rest of the terms,
\begin{eqnarray}{\label{trauptramup}}
& & 2\beta\left( \mbox{Traç}\Upsilon(\xi)+\frac{1}{\sqrt{\gamma}}w_2 \right)\left( \mbox{Traç}\Upsilon(m(\xi))+\frac{1}{\sqrt{\gamma}}V(w_2) \right) = 2\beta\mbox{Traç}\Upsilon(\xi)\mbox{Traç}\Upsilon(m(\xi))\nonumber\\
&+& \frac{2\beta}{\sqrt{\gamma}}\mbox{Traç}\Upsilon(\xi)V(w_2) + \frac{2\beta}{\sqrt{\gamma}}w_2\mbox{Traç}\Upsilon(m(\xi))+ \frac{2\beta}{\sqrt{\gamma}}w_2V(w_2)\nonumber\\
&=& 2\beta\mbox{Traç}\Upsilon(\xi)\left( \mbox{Traç}\nabla_V\Upsilon(\xi) + \mbox{Traç}G(V,DW_1) +Lo(\xi) \right) + \frac{2\beta}{\sqrt{\gamma}}\mbox{Traç}\Upsilon(\xi)V(w_2)\nonumber\\
&+& \frac{2\beta}{\sqrt{\gamma}} w_2\mbox{Traç}\Upsilon(m(\xi)) + \frac{\beta}{\gamma}V(w_2^2)\nonumber\\
&=& 2\beta \mbox{Traç}\Upsilon(\xi)\mbox{Traç}\nabla_V \Upsilon(\xi) + 2\beta \mbox{Traç}\Upsilon(\xi)\mbox{Traç}G(V,DW_1) + \frac{2\beta}{\sqrt{\gamma}} \mbox{Traç}\Upsilon(\xi)V(w_2)\nonumber\\
&+& \frac{2\beta}{\sqrt{\gamma}} w_2 \mbox{Traç}\Upsilon(m(\xi)) + \frac{\beta}{\gamma}V(w_2^2)\nonumber\\
&=& \beta V\left( \left( \mbox{Traç}\Upsilon(\xi)+\frac{1}{\gamma}w_2 \right)^2 \right) +2\beta\mbox{Traç}\Upsilon(\xi)\mbox{Traç}G(V,DW_1) +Lo(\xi)
\end{eqnarray}
and,
\begin{eqnarray}{\label{chichimxi}}
2\left< \chi(\xi),\chi(m(\xi)) \right> &=& 2\left< \chi(\xi),\nabla_{V}\chi(\xi)+G(V,DW_2)+Lo(\xi) \right>\nonumber\\
&=& 2\left< \chi(\xi),\nabla_{V}\chi(\xi) \right> + 2\left< \chi(\xi),G(V,DW_2) \right> + Lo(\xi)\nonumber\\
&=& V\left( |\chi(\xi)|^2 \right) + 2\left< \chi(\xi), G(V,DW_2) \right> + Lo(\xi)
\end{eqnarray}
and,
\begin{eqnarray}{\label{trachitramxi}}
2\beta\mbox{Traç}(\chi(\xi))\beta\mbox{Traç}(\chi(m(\xi))) &=& 2\beta\mbox{Traç}(\chi(\xi))\left(  \mbox{Traç}\nabla_V \chi(\xi) + \mbox{Traç}G(V,DW_2) +Lo(\xi) \right)\nonumber\\
&=& \beta V\left( \left( \mbox{Traç}\chi(\xi) \right)^2 \right) + 2\beta \mbox{Traç}\chi(\xi)\mbox{Traç}G(V,DW_2)+Lo(\xi)\nonumber\\
\end{eqnarray}

substituting (\ref{trachitramxi}), (\ref{chichimxi}), (\ref{trauptramup}), (\ref{upsxiupmxi}), (\ref{DW2DmW2}) and  (\ref{varpivarmxi}) in (\ref{BNEvProv}), we have
\begin{eqnarray*}
\mathbf{B}(\xi,m(\xi)) &=& V\left(|\Upsilon(\xi)|^2 \right) +2 \left< \Upsilon(\xi), G(V,DW_1) \right> + 2V\left( |\varphi(\xi)|^2 \right)+4v|\varphi(\xi)|^2\\
&+& \beta V\left(\left( \mbox{Traç}\Upsilon(\xi)+\frac{1}{\sqrt{\gamma}}w_2 \right)^2 \right) + 2\beta\mbox{Traç}\Upsilon(\xi)\mbox{Traç}G(V,DW_1)\\
&+& 2\beta \mbox{Traç}\Upsilon(\xi)\mbox{Traç} G(V,DW_1)+ V\left( |\chi(\xi)|^2 \right) + 2\left< \chi(\xi),G(V,DW_2) \right> + \beta V\left( \left( \mbox{Traç}\chi(\xi) \right)^2 \right) \\
&+& 2\beta\mbox{Traç}\chi(\xi)\mbox{Traç}G(V,DW_2) + \frac{1}{2}V\left( |Dw_2|^2 \right) + v|Dw_2|^2\\
&+&\frac{1}{\gamma}V(w_2^2) + Lo(\xi)\\
&=& \frac{1}{2}V\left( 2|\Upsilon(\xi)|^2 + 4|\varphi(\xi)|^2 + 2\beta\left( \mbox{Traç}\Upsilon(\xi)+\frac{1}{\sqrt{\gamma}} w_2\right)^2 +2|\chi(\xi)|^2 + 2\beta\left( \mbox{Traç}\chi(\xi) \right)^2 \right)\\
&+&\frac{1}{2}\left( |Dw_2|^2 +\frac{1}{\gamma}w_2^2 \right) + 2\left< \Upsilon(\xi), G(V,DW_1) \right> + 4v|\varphi(\xi)|^2 + 2\beta\mbox{Traç}\Upsilon(\xi)\mbox{Traç}G(V,DW_1)\\
&+& 2\left< \chi(\xi),G(V,DW_2) \right> + 2\beta\mbox{Traç}\chi(\xi)\mbox{Traç}G(V,DW_2) +v|Dw_2|^2 +Lo(\xi)\\
&=& \frac{1}{2}V\left( B(\xi,\xi) \right) + 2b\left( \Upsilon(\xi),G(V,DW_1) \right) + 2b\left( \chi(\xi),G(V,DW_2) \right)\\
&+& 4v|\varphi(\xi)|^2 + v|Dw_2|^2 + Lo(\xi)
\end{eqnarray*}
Then,
\begin{eqnarray}{\label{yana1}}
2\mathsf{B}(\xi,m(\xi))&=&\int_{\Omega}\left[ V\left( B(\xi,xi) \right)+4b\left( \Upsilon(\xi),G(V,DW_1) \right)+ b\left( \chi(\xi),G(V,DW_2) \right) \right]dx\nonumber \\
&+&\int_{\Omega}\left[ 8v|\varphi(\xi)|^2 + 2v|Dw_2|^2 + Lo(\xi) \right]
\end{eqnarray}
Now, being $V=\sum_{i=1}^{2}V_iE_i$, where $\{E_1,E_2\}$ is a normal referential that varies with $x$. have
\begin{eqnarray}{\label{yana2}}
\int_{\Omega}V\left( B(\xi,\xi) \right)dx &=& \sum_{i=1}^{2}\int_{\Omega}V_iE_i(B(\xi,\xi))dx\nonumber\\
&=&\sum_{i=1}^{2}\int_{\Omega}E_i\left( V_i B(\xi,\xi) \right)dx - \sum_{i=1}^{2}\int_{\Omega}E_i(V_i)B(\xi,\xi)dx\nonumber\\
&=&\int_{\Omega}\mbox{div}\left( B(\xi,\xi)V \right)dx -\int_{\Omega}\mbox{div}(V)B(\xi,\xi)dx\nonumber\\
&=&\int_{\Omega}B(\xi,\xi)\left< V,\nu \right> dx -2\int_{\Omega}vB(\xi,\xi)dx 
\end{eqnarray}
In the last line of (\ref{yana2}) we will use (\ref{Camp_Esca_NagEv}). In fact,
\[\mbox{div}(V)=\mbox{Traç}DV=\sum_{i=1}^{2}DV(E_i,E_i)=\sum_{i=1}^{2}v|E_i|^2=2v\]
Substituting (\ref{yana2}) into (\ref{yana1}), we have
\[2\mathsf{B}(\xi,m(\xi))=\int_{\Gamma}B(\xi,\xi)dx-2\int_{\Omega}vB(\xi,\xi)dx + 2\int_{\Omega}e(\xi,\xi)dx+Lo(\xi)\]
with $e(\xi,\xi)=2b\left( \Upsilon(\xi),G(V,DW_1) \right)+b\left( \chi(\xi),G(V,DW_2) \right)+4v|\varphi(\xi)|^2+v|Dw_2|^2$ and the lemma is .
\end{proof}

Other result need is the following

\begin{theorem}{\label{Ident1NagEv}}
Let $\xi=\left(W_1,W_2,w_1,w_2 \right)\in\mathsf{H}^1(\Omega)$ problem solution
\begin{equation}{\label{NaghHomg}}
\xi_{tt}+A\xi+a(x)\xi_t=0,\quad\mbox{em}\quad (0,T)\times\Omega
\end{equation}
Then
\begin{equation}{\label{Ident_1NagEv}}
\begin{aligned}
\int_{\Sigma}\left[ 2\partial(A\xi,m(\xi))+\left( |\xi_t|^2-B(\xi,\xi) \right)\left< V,\nu \right> \right]d\Sigma &= 2\left( \xi_t,m(\xi) \right)|_0^T+2\int_{Q}v\left[ |\xi_t|^2-B(\xi,\xi) \right]dQ\\
&+ 2\int_{Q}e(\xi,\xi)dQ - \int_{0}^{T}\left( a\xi_t,2m(\xi) \right)+ L(\xi)
\end{aligned}
\end{equation}
where $\Sigma=(0,T)\times \Gamma$ and $e(\xi,\xi)=e(\xi,\xi)=2b(S(W_1),G(V,DW_1))+2b(S(W_2),G(V,DW_2))+4v|\varphi(\xi)|^2+v|Dw_2|^2$ . $L(\xi)$denotes the terms of lower order with respect to energy. Additionally, if $p$ is a function in $\Omega$, then
\begin{equation}{\label{Ident_2NagEv}}
\int_{\Sigma}\partial\left( A\xi,p\xi \right)dQ=\int_{Q}p\left[ B(\xi,\xi)-|\xi_t|^2 \right]dQ -\int_{0}^{T}\left(  a\xi_t,p\xi  \right) + L(\xi)
\end{equation}
\end{theorem}
\begin{proof}
Multiplying the equation (\ref{NaghHomg}) by $2m(\xi)$ and integrating in $\Omega$ we have:
\begin{eqnarray}{\label{Fi1}}
\left( \xi_{tt},2m(\xi) \right)_{\mathsf{L}^2(\Omega)}+\left( A\xi,2m(\xi) \right) &=& -\left(  a\xi_t,2m(\xi)  \right)\nonumber\\
\left( \xi_{tt},2m(\xi) \right)_{\mathsf{L}^2(\Omega)}+B(\xi,2m(\xi))-\int_{\Gamma}\partial\left( \xi,2m(\xi) \right)&=&-\left(  a\xi_t,2m(\xi)  \right)
\end{eqnarray}
Estimating the first term on the left hand side of (\ref{Fi1}), since the second term was estimated in the lemma (\ref{BNa_I}).
\begin{equation}{\label{Fi2}}
\left( \xi_{tt},2m(\xi) \right)_{\mathsf{L}^2(\Omega)}=2\left[ \left( \xi_t,m(\xi) \right)_{\mathsf{L}^2(\Omega)} \right]_t -2\left( \xi_t,m(\xi_t) \right)_{\mathsf{L}^2(\Omega)}
\end{equation}
and,
\begin{eqnarray}{\label{Fi3}}
2\left( \xi_{tt},2m(\xi) \right)_{\mathsf{L}^2(\Omega)}&=&2\left( \left( W_{1t},W_{2t},w_{1t},w_{2t} \right),\left( \nabla_V W_{1t},\nabla_V W_{2t},V(w_{1t}),V(w_{2t}) \right) \right)_{\mathsf{L}^2(\Omega)} \nonumber\\
&=&V\left( ||W_{1t}||_{L^2(\Omega,\Lambda)}^{2}+ ||W_{2t}||_{L^2(\Omega,\Lambda)}^{2} +||w_{1t}||_{L^2(\Omega)}^{2}+||w_{2t}||_{L^2(\Omega)}^{2} \right) \nonumber\\
&=&V(||\xi_t||_{L^2(\Omega)}^{2})=\int_{\Omega}V(|\xi_t|^2)dx =\sum_{i=1}^{2}\int_{\Omega}V_iE_i(|\xi_t|^2)dx\nonumber\\
&=&\sum_{i=1}^{2}\int_{\Omega}E_i\left( V_i|\xi_t|^2 \right)dx-\sum_{i=1}^{2}\int_{\Omega}E_i(V_i)|\xi_t|^2 dx \nonumber\\
&=&\int_{\Gamma}|\xi_t|^2\left< V,\nu \right>d\Gamma - 2\int_{\Omega}v|\xi_t|^2 dx
\end{eqnarray}
substituting (\ref{Fi3}) en (\ref{Fi2}), have
\begin{equation}{\label{Fi4}}
\left( \xi_{tt},2m(\xi) \right)_{\mathsf{L}^2(\Omega)}=2 \left[ \left( \xi_t,m(\xi) \right)_{\mathsf{L}^2(\Omega)} \right]_t + 2\int_{\Omega}v|\xi_t|^2 dx -\int{\Gamma}|\xi_t|^2\left< V,\nu \right>d\Gamma
\end{equation}
substituting (\ref{Fi4}) and using the lema (\ref{BNa_I}) in (\ref{Fi1}), have:
\begin{eqnarray}{\label{Fi5}}
2\left[ \left( \xi_{tt},m(\xi) \right)_{\mathsf{L}^2(\Omega)} \right]_t &+& 2\int_{\Omega}v|\xi_t|^2 dx-\int_{\Gamma}|\xi_t|^2\left< V,\nu \right>d\Gamma + \int_{\Gamma}B(\xi,\xi)\left< V,\nu \right>d\Gamma\nonumber\\
&-&2\int_{\Omega}vB(\xi,\xi)dx + 2\int_{\Omega}e(\xi,\xi)dx -\int_{\Gamma}\partial\left( \xi,2m(\xi) \right)d\Gamma +Lo(\xi)=\nonumber\\
&-&\left(  a\xi_t,2m(\xi)  \right)
\end{eqnarray}
integrating (\ref{Fi5}) of $0$ to $T$, have
\begin{eqnarray*}
2\left[ \left( \xi_{tt},m(\xi) \right)_{\mathsf{L}^2(\Omega)} \right]|^{T}_{0} &+& 2\int_{Q}v|\xi_t|^2 dx-\int_{\Sigma}|\xi_t|^2\left< V,\nu \right>d\Sigma + \int_{\Sigma}B(\xi,\xi)\left< V,\nu \right>d\Sigma\nonumber\\
&-&2\int_{Q}vB(\xi,\xi)dx + 2\int_{Q}e(\xi,\xi)dx -\int_{\Sigma}\partial\left( \xi,2m(\xi) \right)d\Sigma +Lo(\xi)=\\
& &\int_{0}^{T}\left(  a\xi_t,2m(\xi)  \right)\nonumber \\
\end{eqnarray*}
where,
\begin{eqnarray*}
\int_{\Sigma}\left[ 2\partial(\xi,m(\xi))+\left( |\xi_t|^2-B(\xi,\xi) \right)\left< V,\nu \right> \right]d\Sigma &=&2\left[ \left( \xi_t,m(\xi_t) \right)_{\mathsf{L}^2(\Omega)} \right]|_{0}^{T}+2\int_{Q}v\left[ |\xi_t|^2-B(\xi,\xi) \right]\nonumber\\
&+& 2\int_{Q}e(\xi,\xi)dQ -\int_{0}^{T}\left(  a\xi_t,2m(\xi)  \right)+ Lo(\xi)
\end{eqnarray*}
And the theorem is shown.
\end{proof}

In this section we state and prove the main result, the stabilization of the evolution equation of the Naghdi model. We will need the following result.
\begin{theorem}{\label{Ident_Observa_Naghdi}}
Let $V$ an escape vector field for the Naghdi shell. Let $\xi$ problem solution.

\begin{equation}{\label{Naghdi_Evolu_Disp}}
\xi_{tt}+A\xi+a(x)\xi_{t}=0.
\end{equation}

Then for $T>0$,

\begin{equation}{\label{Desi_Obs_Dis}}
\int_{\Sigma}SB d\Sigma +\lambda_0\sigma_0\left[ E(0)+e(T) \right] - \int_{0}^{T}\int_{\Omega}a\left< \xi,\xi_t \right> \geq 2\sigma_1 \int_{0}^{T}E(t)dt +L(\xi)
\end{equation}
where

\begin{equation}{\label{Desi_Obs_Dis_Cont}}
\begin{aligned}
SB &= \partial\left(  A\xi,2m(\xi)+\rho\xi  \right) + \left[   \mid  \xi_{t}  \mid^2 - B(\xi,\xi)\right]\left< V,\nu  \right>\\
m(\xi)&= \left(  \nabla_{V}W_1,\nabla_{V}W_2,V(w_1),V(w_2)  \right), \quad \rho=2v-\sigma_1
\end{aligned}
\end{equation}

\end{theorem}

\begin{proof}
Let $p=\rho$ in the identity (\ref{Ident_2NagEv}) and adding with the identity (\ref{Ident_2NagEv}), obtain
\begin{equation}{\label{Naghdi_00}}
	\begin{aligned}
		\int_{\Sigma}SB d\Sigma &= 2\left( \xi_t,m(\xi) \right)_{(L^2(\Omega))}\mid_{0}^{T} + \sigma_1\int_{Q}\left[   \mid \xi_t \mid^2 + B(\xi,\xi)\right]dQ + \\
		& 2\int_{Q}e(\xi,\xi)dQ +L(\xi)
	\end{aligned}
\end{equation}

and, by the expression of $B$, we have to

\begin{equation}{\label{EQ_bb}}
\begin{aligned}
B(\xi,\xi) &= 2b(S(W_1),S(W_1))+ 2 \beta b(S(W_2),S(W_2))+\\
 &4\mid \varphi(\xi)\mid^2 + \mid Dw_2 \mid +L(\xi)
\end{aligned}
\end{equation}

Using the lemma (\ref{Des_b}), have
\begin{equation}{\label{EQ_ee}}
\int_{Q}e(\xi,\xi)dQ \geq \sigma_1 \int_{Q}B(\xi,\xi)dQ +L(\xi)
\end{equation}
using the identity (\ref{EQ_bb}), the coercivity of $b$, have

\begin{equation}{\label{EQ_ss}}
\begin{aligned}
2\left( \xi_t,m(\xi)  \right)_{L^2(\Omega)} &\leq \sigma_0\left[   \parallel \xi_t \parallel^2_{L^2(\Omega)}  +\sum_{i=1}^{2}\left(  \parallel DW_i \parallel^2_{L^2(\Omega,T^2)}+\parallel Dw_i \parallel^2_{L^2(\Omega,\Lambda)} \right)\right] \\
&\leq 2\lambda_0\sigma_0 E(t)+L(\xi)
\end{aligned}
\end{equation}
finally, substituting (\ref{EQ_ss}) and (\ref{EQ_ee}), in (\ref{Naghdi_00}), we have inequality (\ref{Desi_Obs_Dis}).
\end{proof}

We now state and prove our main result, the stabilization of the Naghdi evolution model.

\begin{theorem}{\label{Esta_Nagh}}
Let be the evolution equation for the Naghdi shell model, with internal dissipation
\begin{equation}{\label{Nagh_Dis}}
\begin{aligned}
\xi_{tt}+A\xi+a(x)\xi_t &=0\quad\mbox{em}\quad \Omega\times(0,T)\\
\xi &=0\quad\mbox{em}\quad\partial\Omega\times(0,T)
\end{aligned}
\end{equation}
where the function $a=a(x)$ is supported in an escape region $w\subset\bar{\Omega}$. Let $a_0>0$ such that
\begin{equation}{\label{Fun_a}}
a(x)\geq a_0>0\quad\mbox{para todo}\quad x\in w
\end{equation}
So, there are constants $c_1,c_2$ such that
\begin{equation}{\label{Deca_Exp_Ener}}
E(t)\leq c_1 E(0)\exp^{(-c_2t)}
\end{equation}
where $E(t)$ is the total energy of the system (\ref{Nagh_Dis})
\end{theorem}

\begin{proof}
multiplying the equation (\ref{Nagh_Dis}) for $\xi_t$, integrating into $\Omega$ and considering the boundary conditions, we have that
\begin{equation}{\label{Condi_Suf_Deca}}
\frac{d}{dt}E(t)=-\int_{\Omega}a(x)\mid\xi_t\mid^2dx
\end{equation}
from where, integrating into $(0,T)$, we have

\begin{equation}{\label{Rela_ET_EO}}
E(T)=E(0)-\int_{\Omega}a(x)\mid\xi_t\mid^2dx.
\end{equation}

By (\ref{Rela_ET_EO}) it is enough to prove that there is a $T>0$ e $C>0$, independent of the solutions of the problem (\ref{Nagh_Dis}), such that

\[
E(T)\leq C\int_{Q}a\mid \xi_t \mid^2dQ,
\]

for in this case, substituting in (\ref{Rela_ET_EO}), we have to

\begin{equation}{\label{decaimento_N}}
E(T)\leq \frac{C}{C+1}E(0)
\end{equation}

We affirm that (\ref{Deca_Exp_Ener}) if follow from (\ref{decaimento_N}). In fact, notice that (\ref{decaimento_N}) is equivalent

\begin{equation}{\label{Axx_D}}
E(T)\leq \gamma E(0)
\end{equation}
where $\gamma=\frac{C}{C+1}<1$. Since the system is invariant by translations in time, we have that (\ref{Axx_D})is valid in $[(m-1)T,mT]$, so
\begin{equation}{\label{Ax_D}}
\begin{aligned}
E(mT) &\leq \gamma E((m-1)T)\\
& \leq \gamma^2 E((m-2)T)\\
& \vdots \\
& \leq \gamma^m E(0)\\
& \leq e^{-\omega mT}E(0)
\end{aligned}
\end{equation}
where $\omega=\frac{1}{T}\ln(\frac{1}{\gamma})>0$. For $t>0$ arbitrary, exists $m=1,2,\dots$, such that $(m-1)T<t\leq mT$. Finally, using that the energy of the system is decreasing, we have:
\begin{eqnarray*}
E(t) &\leq & E((m-1)T) \\
&\leq & e^{-\omega (m-1)T}E(0)\\
&\leq & \frac{1}{\gamma}e^{-\omega mT}E(0)\\
&\leq & \frac{1}{\gamma}e^{-\omega t}E(0)
\end{eqnarray*}
Which proves our statement. So, let try (\ref{decaimento_N})

By the definition of escape fields and regions discussed, we can assume that there are subsets of $\Omega$, $\{\Omega_i\}_{i=1}^{N}$, satisfying (\ref{Camp_Esca_NagEv}). Then the identities (\ref{Desi_Obs_Dis_Cont}) and (\ref{Ident_2NagEv}) can be used on each $\Omega_i$, since vectorial escape fields are defined on each of them.

The idea is to estimate the total energy of the system, first inside $\Omega$ using that on the $\Omega_i$ we have vanishing vector fields defined on them and in the plugin use the property of the dissipation function, $a$.

Now, since (\ref{Desi_Obs_Dis}) is valid only in $\Omega_i$, We are going to restrict ourselves, first, to making estimates on them.

be then $0<\epsilon_2<\epsilon_1<\epsilon_0<\epsilon$ and be
\[V_j=N_{\epsilon_j}\left\{ \cup_{i=1}^{N}\Gamma_{0}^{i}\cup\left( \Omega\setminus\cup_{i=1}^{N}\Omega_i \right) \right\},\quad j=0,1,2.\]
notice that $V_2\subset V_1\subset V_0\subset \bar{V_0}$. be
\[
\phi^i = \left\{\begin{array}{ccc}
1, & \Omega_i\setminus V_i, & i=1,2,\dots,N\\
0, & V_2 
\end{array} \right.
\]
consider now $V^i=\phi^i H^i$, $p^i=\phi^i q$ and $\xi^i=\phi^i\xi$ where $H^i$ is an escape vector field in $\Omega_i$ and $q$ a function defined in $\Omega$, respectively. So, in each $\Omega_i$ it is valid (\ref{Desi_Obs_Dis}) and we have

\begin{equation}{\label{EQ1_DemF}}
\begin{aligned}
2\sigma_1\int_{0}^{T}E^i(t)\leq \int_{\Sigma_i}SB_i d\Sigma_i +\lambda_0\sigma_0\left[ E^i(0)+E^i(T)  \right] - \int_{0}^{T}\int_{\Omega_i}a\left< \xi^i,\xi^i_t \right>dxdt
\end{aligned}
\end{equation}
where,
\begin{equation}{\label{Interface_Naghdi}}
SB_i = \partial\left(  A\xi^i,2m(\xi^i)+\rho\xi^i  \right) + \left[   \mid  \xi^i_{t}  \mid^2 - B(\xi^i,\xi^i)\right]\left< V^i,\nu  \right>
\end{equation}

By the definition of $V_2$, we have to $\Gamma^{i}_{0}\subset V_2$ and what $\phi^i=0$ in $V_2$, we have to $x\in\Gamma^{i}_{0}$, the terms at the border are annulled. So

\begin{equation}{\label{Cond_Gi0}}
SB_i=0\quad\quad \mbox{para}\quad x\in \Gamma^{i}_{0}.
\end{equation}

If $x\in \Omega_i\cap \Gamma$, using boundary conditions, $\xi=0$, replacing in (\ref{Interface_Naghdi}) have

\begin{equation}{\label{FF_1}}
\begin{aligned}
SB_i &=2\partial\left(  A\xi_i,m(\xi^i) \right)-B(\xi^i,\xi^i)\left< V^i,\nu  \right>\\
&=B(\xi^i,\xi^i)\left< V^i,\nu  \right>\\
&\leq 0
\end{aligned}
\end{equation}
where we use the coercivity of $B$. It is concluded that
\begin{equation}{\label{FF_2}}
\int_{\Sigma_i}SB_i\leq 0\quad\mbox{para todo}\quad i
\end{equation}

Using the estimate (\ref{Cond_Gi0}) and (\ref{FF_2}) in  (\ref{EQ1_DemF}) have

\begin{equation}{\label{PP0}}
\begin{aligned}
2\sigma_1\int_{0}^{T}\int_{\Omega_i}\mid \xi_t \mid^2 + 2\sigma_1\int_{0}^{T}\int_{\Omega_i} B(\xi,\xi)\leq -\int_{0}^{T}\left(  a\xi_t,2m(\xi)  \right)+\lambda_0\sigma_0\left[  E(T)+E(0)\right] +L(\xi)\nonumber
\end{aligned}
\end{equation}
where,

\begin{equation}{\label{PP0}}
\begin{aligned}
\int_{0}^{T}\int_{\Omega_i\setminus V_1} B(\xi,\xi) &\leq C_1\int_{T}^{0}\int_{\Omega_i\cap V_1}B(\xi,\xi)+C_{\beta}\int_{0}^{T}\int_{\Omega_i}\mid \xi \mid^2 + \beta\int_{0}^{T}\int_{\Omega_i}B(\xi,\xi)dQ +\\
&\lambda_0\sigma_0\left[  E(T)+E(0)\right] +L(\xi)
\end{aligned}
\end{equation}
where $\beta>0$ is small enough that $\beta \int_{0}^{T}\int_{\Omega_i} B(\xi,\xi)dxdt\leq \int_{0}^{T}\int_{\Omega_i\cap V_1} B(\xi,\xi)dxdt$.

Given that $\Omega\subset \left( \cup\Omega_i \right)\cup V_1$, then

\begin{equation}{\label{PP1}}
\begin{aligned}
\int_{0}^{T}\int_{\Omega\setminus V_1} B(\xi,\xi)&\leq\sum_{i=1}^{N}\int_{0}^{T}\int_{\Omega_i\setminus V_1} B(\xi,\xi)\\ &\leq C_2\int_{0}^{T}\int_{\Omega\cap V_1}B(\xi,\xi)dQ 
& + C_3\int_{0}^{T}a\parallel \xi_t\parallel^2 dt+ \lambda_0\sigma_0\left[  E(T)+E(0)\right] +L(\xi) 
\end{aligned}
\end{equation}

Now let us estimate in the complement of the union of the $\Omega_i$. For this, be $\psi\in C^{\infty}(I\!\!R^{3})$ given by

\begin{equation}{\label{Fun_Bordo}}
\psi(x)=\left\{\begin{aligned}
0 &, \quad x\in I\!\!R^3\setminus V_0 \\
1 &, \quad x\in V_1
\end{aligned}\right.
\end{equation}

Considering $p=\psi$ en (\ref{Ident_2NagEv}) have
\[
\int_{Q_i}B(\xi,\xi)dQ_i = \int_{Q_i}\mid \xi_t\mid^2 dQ_i + \int_{0}^{T}\left( a\xi_t,\psi\xi  \right)dt +L(\xi)
\]
where from,
\begin{equation}{\label{Estima_bordo}}
\begin{aligned}
\int_{0}^{T}\int_{\Omega\cap V_1}B(\xi,\xi)dQ + \int_{0}^{T}\int_{\Omega\cap V_0}B(\xi,\xi)dQ &\leq  \int_{0}^{T}\int_{\Omega\cap V_0}B \parallel \xi_t \parallel^2 dQ + \int_{0}^{T}\int_{\Omega\cap V_0} \left<a\xi_t,\xi\right> +L(\xi)\nonumber
\end{aligned}
\end{equation}
And what $\epsilon_0<\epsilon$, then $w\supset \bar{\Omega}\cap V_0$. Using the function hypothesis $a$ in $w$, have

\begin{equation}{\label{Estima_bordo_1}}
\begin{aligned}
\int_{0}^{T}\int_{\Omega\cap V_1}B(\xi,\xi)dQ &\leq  \int_{0}^{T}\int_{\Omega\cap V_0}\parallel \xi_t \parallel^2 dQ + \int_{0}^{T}\int_{\Omega\cap V_0} \left<a\xi_t,\xi\right> +L(\xi)\nonumber\\
& \leq \frac{1}{a_0}\int_{0}^{T}\int_{\Omega\cap V_0}a \mid\xi\mid^2 dt + L(\xi)+\beta\int_{0}^{T}\parallel\xi_t\parallel^2 dt+L(\xi)
\end{aligned}
\end{equation}
where $\beta$ will be chosen later. so, of (\ref{PP1}), (\ref{Estima_bordo_1}), have

\begin{equation}{\label{Estima_BB}}
\begin{aligned}
\int_{Q}B(\xi,\xi) &= \int_{0}^{T}\int_{\Omega\setminus V_1} B(\xi,\xi)dx dt  + \int_{0}^{T}\int_{\Omega\cap V_1} B(\xi,\xi)dx dt \\
&\leq C_4\int_{0}^{T}\int_{\Omega\cap V_1} B(\xi,\xi)dx dt +C_5\int_{0}^{T}a\parallel\xi_t\parallel^2 dt +L(\xi) \\
&\leq C_6\int_{Q} a\mid \xi_t\mid^2 dt + \beta\int_{0}^{T}\parallel\xi_t\parallel^2 dt + \lambda_0\sigma_0\left[  E(T)+E(0)\right]+ L(\xi)
\end{aligned}
\end{equation}

now considering $p=\frac{1}{2}$ in (\ref{Ident_2NagEv}), have
\begin{equation}{\label{Pmedio}}
\begin{aligned}
\frac{1}{2}\int_{Q} \left[ \mid \xi_t\mid^2-B(\xi,\xi) \right]dQ &= \frac{1}{2}\int_{0}^{T}\left(  a\xi_t,\xi  \right)dt +L(\xi)
&\leq C_8\int_{0}^{T}a\parallel \xi_t  \parallel^2 dt +L(\xi)
\end{aligned}
\end{equation}

grouping (\ref{Estima_BB}) and (\ref{Pmedio}), have
\begin{equation}{\label{Energia_final}}
\begin{aligned}
\int_{0}^{T}E(t)dt &= \int_{Q}B(\xi,\xi)dQ +\frac{1}{2}\int_{Q}\left[  \mid \xi_t \mid^2 - B(\xi, \xi) \right]dQ \\
&\leq C\int_{Q}a\mid \xi_t \mid^2 dQ +\beta \int_{0}^{T} \parallel \xi_t \parallel^2 dt + \lambda_0\sigma_0\left[  E(T)+E(0)\right]+C_8 \int_{Q}a\mid \xi_t\mid dQ +L(\xi)\\
&\leq C_9\int_{Q}a\mid \xi_t\mid dQ + \beta \int_{0}^{T}E(t)dt + C_9\left[E(T)+E(0)\right] + L(\xi)
\end{aligned}
\end{equation}

Taking $\beta=\frac{1}{2}$, considering that the energy is decreasing, this is, $E(t)\geq E(T)$ for $t\in [0,T]$ and $E(T)=E(0)-\int_{Q}a\mid \xi \mid^2 dQ$, have 

\begin{equation}{\label{FFF}}
\begin{aligned}
\frac{1}{2}\int_{0}^{T}E(t)dt &\leq C_9\int_{Q}a\mid \xi_t\mid dQ + C_9\left[E(T)+E(0)\right] + L(\xi)\\
\frac{T}{2} E(T) & \leq C_9\int_{Q}a\mid \xi_t\mid dQ + \lambda_0\sigma_0\left[  E(T)+E(T)+\int_{Q}a\mid \xi_t\mid dQ \right] + L(\xi)\\
\left( \frac{T}{2}-2\lambda_0\sigma_0  \right) & \leq C_{10}\int_{Q}a\mid \xi_t\mid dQ +L(\xi)
\end{aligned}
\end{equation}
For $T>4\lambda_0\sigma_0$ of (\ref{FFF}) and the uniqueness-compactness argument \cite{Yao}, have
\[
E(t)\leq C\int_{Q}a\mid \xi_t\mid dQ
\]
This proves the theorem.
\end{proof}

\section{Controllability via Stability}

Russell Principle\cite{Russell_2} provides a method for testing exact controllability using the uniform stabilization result.

In this section we are going to study the exact controllability for the Naghdi evolution system. For this, we apply Russell Principle using the result of energy decay proved in the theorem \ref{Esta_Nagh}

The exact controllability problem, with controls inside, consists of finding a vector function $F=F(x,t)$, control call, such that for some $T>0$ the following problem has a solution:
\begin{equation}{\label{EqCont_1}}
\left\{
\begin{aligned}
& \eta_{tt}+A\eta=F(x,t)\quad\mbox{em}\quad \Omega\times(0,T)\\
& \eta =0\quad\mbox{em}\quad\partial\Omega\times(0,T) \\
& \eta(0)=\eta_{0},\quad\eta_{t}(0)=\eta_1\\
& \eta(T)=\tilde{\eta}_0,\quad\eta_{t}(T)=\tilde{\eta}_1
\end{aligned}\right.
\end{equation}
for any initial data $(\eta_0,\eta_1)$ and endings $(\tilde{\eta}_0,\tilde{\eta}_1)$ in the appropriate functional spaces and $F$ acting in a subregion of $\Omega$.

We will prove that the function $F$ needs to act only in a sub-region of $\Omega$ arbitrarily small, precisely in the escape region for the Naghdi shell.

We consider $T>0$ such that
\begin{equation}{\label{Menor_1}}
c_1 e^{-c_2 T}<1
\end{equation}
where $c_1>0$ and $c_2>0$ are the constants that appear in the theorem \ref{Esta_Nagh}.

In this section we will prove the following result:
\begin{theorem}{\label{Control_exata_Naghi}}
Let $\Omega$ the median surface of the Naghdi shell and $w\subset\Omega$ the vanishing region given in the proof of the theorem \ref{Esta_Nagh}. Then, if $T>0$ satisfies (\ref{Menor_1}) the system (\ref{EqCont_1}) is exactly controllable at time $T$ with controls located at $w$
\end{theorem}

\begin{proof}
Since the system is linear and reversible in time, it is enough to consider the controllability at zero, that is, $(\tilde{\eta}_0, \tilde{\eta}_1)=(0,0)$ 

We have to, for $(\eta_0,\eta_1)\in V\times H$, there is only one solution $\eta\in C\left( [0,\infty); V\right)\times C^1\left( [0,\infty); H\right)$ of the problem.
\begin{equation}{\label{EqCont_2}}
\left\{
\begin{aligned}
& \eta_{tt}+A\eta+a(x)\eta_t=0\quad\mbox{em}\quad \Omega\times(0,T)\\
& \eta =0\quad\mbox{em}\quad\partial\Omega\times(0,T) \\
& \eta(0)=\eta_{0},\quad\eta_{t}(0)=\eta_1\\
\end{aligned}\right.
\end{equation}

Additional, for the initial data $(-\eta(T),\eta_t(T))$ there is only one solution to the problem.
\begin{equation}{\label{EqCont_3}}
\left\{
\begin{aligned}
& \theta_{tt}+A\theta+a(x)\theta_t=0\quad\mbox{em}\quad \Omega\times(0,T)\\
& \theta =0\quad\mbox{em}\quad\partial\Omega\times(0,T) \\
& \theta(0)=-\eta(T),\quad \theta_t(T)=\eta_{t}(T)\\
\end{aligned}\right.
\end{equation}
where $a$ acts in the escape region given in the theorem \ref{Esta_Nagh}.

let define $\xi(x,t)=\eta(x,t)+\theta(x,T-t)$. The field $\xi$ satisfies
\begin{equation}{\label{EqCont_4}}
\left\{
\begin{aligned}
& \xi_{tt}+A\xi=a(x)(\eta_t+\theta_t)\quad\mbox{em}\quad \Omega\times(0,T)\\
& \xi =0\quad\mbox{em}\quad\partial\Omega\times(0,T) \\
& \xi(0)=\eta_{0}+\theta(T),\quad\xi_{t}(0)=\eta_1-\theta_t(T)\\
& \xi(T)=0,\quad\xi_t(T)=0
\end{aligned}\right.
\end{equation}

Of (\ref{EqCont_4}) we have that the initial data that are brought to equilibrium have the form
\begin{equation}{\label{Dados_Equilibro}}
\left(\xi_0,\xi_1\right)=\left( \eta_0+\theta(T),\eta_1 -\theta_t(T) \right)
\end{equation} 
Thus it is enough to prove that for each initial data $\left(\xi_0,\xi_1\right)\in V\times H$, exists $\left(\eta_0,\eta_1\right)$ satisfying (\ref{Dados_Equilibro}). Equivalently, show that the application:
\begin{equation}{\label{Apli_Inv}}
\begin{aligned}
 L:& V\times H \longrightarrow V\times H \\
&\left(\eta_0,\eta_1\right)\longrightarrow \left( \eta_0 +\theta(T),\eta_1 - \theta_t(T) \right)
\end{aligned}
\end{equation}
is surjective. As $L=I-K$ where $K$ is the application given by
\[
K(\eta_0,\eta_1)=\left( -\theta(T),\theta_t(T) \right)
\]
it is enough to show that $\parallel K \parallel_{V\times H}<1$. Applying the theorem twice \ref{Esta_Nagh} have
\begin{equation}
\begin{aligned}
\parallel K(\eta_0,\eta_1) \parallel_{V\times H} &= \parallel  \left( -\theta(T),\theta_t(T) \right) \parallel_{V\times H}\\
&\leq c_1 e^{-c_2T}\parallel \left( -\eta(T),\eta_t(T) \right) \parallel_{V\times H}\\
&\leq c_3e^{-c_4T}\parallel \left( \eta(0),\eta_1 \right) \parallel_{V\times H}.
\end{aligned}
\end{equation}
choosing $T>0$ such that $c_3e^{-c_4T}<1$ we have to $\parallel K \parallel_{V\times H}<1$. Thus $L=I-K$ is a surjective application and therefore the theorem \ref{Control_exata_Naghi} it is demonstrated.
\end{proof}

\section{conclusions}
	\begin{itemize}
		\item[(a)] The dissipation can be considered in an arbitrarily small region of the shell and obtain stabilization and controllability.
		\item[(b)] The existence of flight regions are verifiable for Naghdi shells, this implies that localized dissipation in the complement of the union of said regions generates stabilization and controllability.
	\end{itemize}

\bibliographystyle{amsalpha}

\end{document}